\documentclass[12pt,psamsfonts,reqno]{amsart}
\usepackage{amssymb}
\usepackage{amsfonts}
\usepackage{latexsym}
\usepackage{amscd}
\usepackage[mathscr]{euscript}
\usepackage{xy} \xyoption{all}

\newcommand{\N}{{\mathbb{N}}}

\newcommand{\Z}{{\mathbb{Z}}}

\newcommand{\ol}{\overline}
\newcommand{\uloopr}[1]{\ar@'{@+{[0,0]+(-4,5)}@+{[0,0]+(0,10)}@+{[0,0] +(4,5)}}^{#1}}
\newcommand{\uloopd}[1]{\ar@'{@+{[0,0]+(5,4)}@+{[0,0]+(10,0)}@+{[0,0]+ (5,-4)}}^{#1}}
\newcommand{\dloopr}[1]{\ar@'{@+{[0,0]+(-4,-5)}@+{[0,0]+(0,-10)}@+{[0, 0]+(4,-5)}}_{#1}}
\newcommand{\dloopd}[1]{\ar@'{@+{[0,0]+(-5,4)}@+{[0,0]+(-10,0)}@+{[0,0 ]+(-5,-4)}}_{#1}}

\newcommand{\dnw}{\mathbin{\downarrow}}

\newcommand{\luloop}[1]{\ar@'{@+{[0,0]+(-8,2)}@+{[0,0]+(-10,10)}@+{[0, 0]+(2,2)}}^{#1}}

\newtheorem{lem}{Lemma}[section]
\newtheorem{corol}[lem]{Corollary}
\newtheorem{theor}[lem]{Theorem}
\newtheorem{prop}[lem]{Proposition}
\newtheorem{rema}[lem]{Remark}
\newtheorem{defi}[lem]{Definition}

\newtheorem{exem}[lem]{Example}

\newtheorem{point}[lem]{}

\newtheorem*{nota}{Notation}

\newcommand{\les}{\leqslant}

\newcommand{\AC}[2]{G_{#1}[#2]}

\newcommand{\Ma}{{\rm Max}}

\DeclareMathOperator{\rL}{L}

\numberwithin{equation}{section}

   \usepackage{color}

   \usepackage{color}

\begin{document}
\title[Representing finitely generated refinement monoids]{Representing finitely generated refinement monoids as graph monoids}
\author{Pere Ara}
\address{Departament de Matem\`atiques, Universitat Aut\`onoma de Barcelona,
08193 Bellaterra (Barcelona), Spain.} \email{para@mat.uab.cat}
\author{Enrique Pardo}
\address{Departamento de Matem\'aticas, Facultad de Ciencias, Universidad de C\'adiz,
Campus de Puerto Real, 11510 Puerto Real (C\'adiz),
Spain.}\email{enrique.pardo@uca.es}\urladdr{https://sites.google.com/a/gm.uca.es/enrique-pardo-s-home-page/}

\thanks{Both authors are partially supported by the DGI-MINECO and European Regional Development Fund, jointly, through Project MTM2014-53644-P. The second author was partially supported by PAI III grant FQM-298 of the Junta de Andaluc\'{\i}a. } 
\subjclass[2010]{Primary 06F20, Secondary 16D70, 19K14, 20K20, 46L05, 46L55} 
\keywords{Refinement monoid, regular monoid, primely generated monoid, graph monoid.}
%

\maketitle

\begin{abstract} Graph monoids arise naturally in the study of non-stable K-theory of graph C*-algebras and Leavitt path
algebras. They play also an important role in the current approaches to the realization problem for von Neumann regular rings.
In this paper, we characterize when a finitely generated conical refinement monoid can be represented as a graph monoid.
The characterization is expressed in terms of the behavior of the structural maps of the associated $I$-system at the free primes of the
monoid.
\end{abstract}


\section{Introduction}\label{Introduction}

The class of commutative monoids satisfying the Riesz refinement property --refinement monoids for short-- has been largely studied over the last decades 
in connection with various problems such as non-stable K-Theory of rings and $C^*$-algebras (see e.g. \cite{Aposet, AGOP, AMFP, GPW, PaWe}), classification of Boolean 
algebras (see e.g. \cite{Ket}, \cite{Pierce}), or its own structure theory (see e.g. \cite{Brook, Dobb84, Whr3}).

An important 
invariant in non-stable K-theory is the commutative monoid $\mathcal V(R)$ associated to any ring $R$, consisting of the isomorphism classes of finitely generated projective (left, say) 
$R$-modules, with the operation induced from direct sum. If $R$ is a (von Neumann) regular ring or a C*-algebra with real rank zero (more generally, an exchange ring), 
then $\mathcal V(R)$ is a refinement monoid (e.g., \cite[Corollary 1.3, Theorem 7.3]{AGOP}). 

The \emph{realization problem} asks which refinement monoids appear as a $\mathcal V(R)$ for $R$ in 
one of the above-mentioned classes. Wehrung \cite{W98IsraelJ} constructed a conical refinement monoid of cardinality $\aleph_2$ which is not isomorphic to $\mathcal V(R)$ for any regular 
ring $R$, but it is an important open problem, appearing for the first time in \cite{directsum}, to determine whether every countable conical refinement monoid can be realized 
as $\mathcal V(R)$ for some regular ring $R$. See \cite{Areal} for a survey on this problem, and \cite{AGtrans} for some recent progress on the problem, with connections with the Atiyah Problem.

An interesting situation in which the answer to the realization problem is affirmative is the following:

\begin{theor}[{\cite[Theorem 4.2, Theorem 4.4]{AB}}]
\label{thm:AB-result}
Let $E$ be a row-finite graph, let $M(E)$ be its graph monoid, and let $K$ be any field. Then there exists a (not necessarily unital) von Neumann regular $K$-algebra $Q_K(E)$ such that $\mathcal V(Q_K(E))\cong M(E)$.
\end{theor}

Thus, an intermediate step that could be helpful to give an answer to the realization problem is to characterize which conical refinement monoids are representable as graph monoids. 
The first author, Perera and Wehrung gave such a characterization in the concrete case of finitely generated antisymmetric refinement monoids  \cite[Theorem 5.1]{APW08}. These monoids 
are a particular case of primely generated refinement monoids (see e.g. \cite{Brook}). Recall that an element $p$ in a monoid $M$ is a {\it prime element} if $p$ is not invertible in $M$, and, whenever 
$p\leq a+b$ for $a,b\in M$, then either $p\leq a$ or $p\leq b$ (where $x\le y$ means that $y=x+z$ for some $z\in M$). The monoid $M$ is {\it primely generated} if every non-invertible element of $M$ 
can be written as a sum of prime elements. Primely generated refinement monoids enjoy important cancellation properties, such as separative cancellation and 
unperforation, as shown by Brookfield in \cite[Theorem 4.5, Corollary 5.11(5)]{Brook}. Moreover, it was shown by Brookfield that any finitely 
generated refinement monoid is automatically primely generated \cite[Corollary 6.8]{Brook}.

In \cite{AP}, the authors of the present paper showed that any primely generated refinement monoid can be represented, 
up to isomorphism, as the monoid associated to an $I$-system (a sort of semilattice of cancellative semigroups defined over a suitable 
poset $I$). This result generalizes the representation of two well-known classes of monoids:
\begin{enumerate}
\item The class of {\it primitive monoids}, i.e. antisymmetric
primely generated refinement monoids, see \cite{Pierce}. These monoids are described by means of a set $I$ endowed with an antisymmetric transitive relation. 
\item The class of primely generated conical {\it regular} refinement monoids.
These monoids were characterized by Dobbertin in \cite{Dobb84} in terms of partial orders of abelian groups.
\end{enumerate}

In the present paper, we will benefit of the picture developed in \cite{AP} to state a characterization of monoids representable 
as graph monoids, in the case of finitely generated conical refinement monoids (see Theorem \ref{thm:main}). The condition, that relies 
on the behavior of free primes in the representation by $I$-systems, generalizes \cite[Theorem 5.1]{APW08}.  

A main device used to obtain realization results for refinement monoids by regular rings or C*-algebras of real rank zero has been the consideration of algebras
associated to graphs. Of particular importance has been the use of graph C*-algebras (see e.g. \cite{KPR}, \cite{KPRR}) and of Leavitt path algebras (see e.g. \cite{AA1},
\cite{AAS}, \cite{AMFP}). Indeed, the von Neumann regular algebra $Q_K(E)$ appearing in Theorem \ref{thm:AB-result} is a specific universal localization of the Leavitt path 
algebra $L_K(E)$ associated to the row-finite graph $E$. Some other algebras associated to graphs are also important in further constructions of von Neumann regular rings,
as the first author showed in \cite{Aposet}. In that paper, given any finite poset $\mathbb P$, a von Neumann regular algebra $Q_K(\mathbb P)$ is constructed so that 
$\mathcal V (Q_K( \mathbb P))\cong M(\mathbb P)$, where $M(\mathbb P )$ is the primitive monoid with associated poset of primes $\mathbb P$ and with all primes being free.
In the present paper, Leavitt path algebras play an instrumental role, being essential in determining the condition which characterizes finitely generated graph monoids (see 
Proposition \ref{thm:free-charac}). It is interesting to observe that, as shown in that Proposition, the obstruction to realize an arbitrary finitely generated conical 
refinement monoid as a graph monoid is of K-theoretical nature. 
The solution of this K-theoretical problem was one of the keys to obtain the main result of \cite{Aposet}
(see \cite[Theorem 3.2]{Aposet}), and will
surely play a vital role in the forthcoming approaches to the resolution of the realization problem for finitely generated refinement monoids.

\smallskip

The paper is organized as follows. In Section \ref{Sect:Basics}, we recall all the definitions and results that will be necessary to follow the contents of 
the subsequent sections. In Section \ref{Sect:ReprFinGen}, we study the special case of finitely generated \emph{regular} refinement monoids, and we prove that 
every finitely generated conical regular refinement monoid can be represented as a graph monoid (Theorem \ref{Th:FinGenRegRefAreRepr}). 
In the short Section \ref{Sect:Example}, we use the techniques developed in Section \ref{Sect:ReprFinGen} to offer an easy presentation of the monoid $M=   \Z^+ \cup \{\infty \}$ as a graph monoid. 
This was done in a somewhat more involved way in \cite[Example 6.5]{APW08}. In Section \ref{TopFree}, we prove the main result of the paper, obtaining the characterization of the finitely generated conical 
refinement monoids which are graph monoids. The result is stated in terms of the theory of $I$-systems established in \cite{AP}, and concerns the behavior of the maps associated to the $I$-system at 
the free primes (see Theorem \ref{thm:main}). Finally we use our main result to recover the characterization of finitely generated primitive graph monoids obtained in \cite{APW08}.


\section{Preliminaries}\label{Sect:Basics}

In this section, we will recall the definitions and results necessary to follow the contents of the paper in a self-contained way. We divide this section in four parts.


\subsection{Basics on commutative monoids.}
All semigroups and monoids considered in this paper are commutative. We will denote by $\N$ the semigroup of positive integers, and by $\Z^+$ the monoid of non-negative integers.  

Given a commutative monoid $M$, we set $M^*:=M\setminus\{0\}$. We say that $M$ is {\it conical} if $M^*$ is a semigroup, that is, if, for all $x$, $y$ in $M$, $x+y=0$ only when $x=y=0$. 

Given a monoid $M$, the antisymmetrization $\ol{M}$  of $M$ is the quotient monoid 
of $M$ by the congruence given by $x\equiv y$ if and only if $x\leq y$ and $y\leq x$ (see \cite[Notation 5.1]{Brook}). We will denote the class of an element $x$ of $M$ in $\ol{M}$ by $\ol{x}$.

We say that a monoid $M$ is {\it separative} 
provided $2x=2y=x+y$ always implies $x=y$; there are a number of equivalent formulations of this property, see e.g. \cite[Lemma 2.1]{AGOP}.
We say that  $M$ is
a {\it refinement monoid} if, for all $a$, $b$, $c$, $d$ in
$M$ such that $a+b=c+d$, there exist $w$, $x$, $y$, $z$ in $M$ such
that $a=w+x$, $b=y+z$, 
$c=w+y$ and $d=x+z$.  It will often be convenient to present this
situation in the form of a diagram, as follows:
$$\mbox{\begin{tabular}{|l|l|l|}
\cline{2-3}
\multicolumn{1}{l|}{} & ${c}$ & ${d}$ \\ \hline
${a}$ & ${w}$ & ${x}$ \\ \hline
${b}$ & ${y}$ & ${z}$ \\ \hline
\end{tabular}}.$$  
A basic example of refinement monoid is the monoid $M(E)$ associated to a countable row-finite graph $E$ \cite[Proposition 4.4]{AMFP}.

If $x, y\in M$, we write $x\leq y$ 
if there exists $z\in M$  such that $x+z = y$.
Note that $\le$ is a translation-invariant pre-order on $M$, called the {\it algebraic pre-order} of $M$. All inequalities in commutative monoids will be with respect to this pre-order. 
An element $p$ in a monoid $M$ is a {\it prime element} if $p$ is not invertible in $M$, and, whenever 
$p\leq a+b$ for $a,b\in M$, then either $p\leq a$ or $p\leq b$. The monoid $M$ is {\it primely generated} if every non-invertible element of $M$ 
can be written as a sum of prime elements.
 
An element $x\in M$ is {\it regular} if $2x\leq x$. An element $x\in M$ is an {\it idempotent} if $2x= x$. An element 
$x\in M$ is {\it free} if $nx\leq mx$ implies $n\leq m$. Any element of a separative monoid is either free or regular.
In particular, this is the case for any primely generated refinement monoid, by \cite[Theorem 4.5]{Brook}.

A subset $S$ of a monoid $M$ is called an {\it order-ideal} if $S$ is a subset of $M$ containing $0$,
closed under taking sums and summands within $M$.  An order-ideal can also be described as a submonoid $I$ of
$M$, which is hereditary with respect to the canonical pre-order
$\le $ on $M$: $x\le y$ and $y\in I$ imply $x\in I$. A non-trivial monoid is said to be {\it simple} if it has no non-trivial order-ideals.

If $(S_k)_{k\in \Lambda}$ is a family of (commutative) semigroups, $\bigoplus _{k\in \Lambda} S_k$ (resp. $\prod _{k\in \Lambda} S_k$) 
stands for the coproduct (resp. the product) of the semigroups $S_k$, $k\in \Lambda$, in the category of commutative
semigroups. If the semigroups $S_k$ are subsemigroups of a semigroup $S$, we will denote by $\sum_{k\in \Lambda} S_k$ the subsemigroup of $S$ generated by $\bigcup_{k\in \Lambda}S_k$.
Note that $\sum_{k\in \Lambda} S_k$ is the image of the canonical map $\bigoplus_{k\in \Lambda} S_k\to S$. We will use the notation $\langle X\rangle $ to denote the semigroup
generated by a subset $X$ of a semigroup $S$.

Given a semigroup $M$, we will denote by $G(M)$ the Grothendieck group of $M$. There exists a semigroup homomorphism $\psi_M\colon M\to G(M)$
such that for any semigroup homomorphism $\eta \colon M\to H$ to a group $H$ there is a unique group homomorphism $\widetilde{\eta}\colon G(M)\to H$ such that 
$\widetilde{\eta}\circ \psi_M= \eta$. $G(M)$ is abelian and it is generated as a group by $\psi (M)$.  If $M$ is already a group then $G(M)= M$. If $M$ is a semigroup of the form 
$\N\times G$, where $G$ is an abelian group, then $G(M)= \Z\times G$. In this case, we will view $G$ as a subgroup of $\Z\times G$ by means of the identification $g\leftrightarrow (0,g)$.

Let $M$ be a conical commutative monoid, and let $x\in M$ be any element. The {\it archimedean component} of $M$ generated by $x$ is the subsemigroup
$$G_M[x]:=\{a\in M : a\leq nx \text{ and } x\leq ma \text{ for some } n,m\in \N\}.$$

For any $x\in M$, $G_M[x]$ is a simple semigroup. If $M$ is separative, then $G_M[x]$ is a cancellative semigroup; if moreover $x$ is a regular element, 
then $G_M[x]$ is an abelian group.


\subsection{Primely generated refinement monoids}
\label{subsec:Primely-gen}

The structure of primely generated refinement monoids has been recently described in \cite{AP}. We recall here some basic facts.

Given a poset $(I, \leq)$, we say that a subset $A$ of $I$ is a {\it lower set} if $x\leq y$ in $I$ and $y\in A$ implies $x\in A$. For any $i\in I$, 
we will denote by $I\downarrow i=\{x\in I : x\leq i\}$ the lower subset generated by $i$. 
We will write $x<y$ if $x\le y$ and $x\ne y$. 

The following definition \cite[Definition 1.1]{AP} is crucial for this work:

\begin{defi}[{\cite[Definition 1.1]{AP}}]
\label{def:I-system} {\rm Let $I= (I,\le )$ be a poset.  An {\it $I$-system} $$\mathcal{J}=
\left(I, \leq , (G_i)_{i\in I}, \varphi_{ji} \, (i<j)\right)$$ is given by the following data:
\begin{enumerate}
\item[(a)] A partition
$I=I_{free}\sqcup I_{reg}$ (we admit one of the two sets
$I_{free}$ or $I_{reg}$ to be empty).
\item[(b)] A family $\{G_i\}_{i\in I}$ of abelian groups. We adopt the following notation: 
\begin{itemize}
\item[(1)] For $i\in I_{reg}$, set $M_i = G_i$, and $\widehat{G}_i=G_i=M_i$.
\item[(2)] For $i\in I_{free}$, set $M_i=\N \times G_i$, and $\widehat{G}_i= \Z\times G_i$
\end{itemize}
Observe that, in any case, $\widehat{G}_i$ is the Grothendieck group of $M_i$.
\item[(c)]
A family of  semigroup homomorphisms $\varphi _{ji}\colon M_i\to
G_j$ for all $i<j$, to which we associate, for all $i<j$, the unique extension $\widehat{\varphi}_{ji}\colon \widehat{G}_i \to G_j$ of $\varphi _{ji}$ to a group homomorphism
from the Grothendieck group of $M_i$  to $G_j$ (we look at these
maps as maps from $\widehat{G}_i$ to $\widehat{G}_j$). We require that the family $ \{ \varphi_{ji} \}$ satisfies the following conditions:
\begin{itemize}
\item[(1)]  The assignment
$$
\left\{
\begin{array}{ccc}
i & \mapsto  &  \widehat{G}_i   \\
 (i<j) & \mapsto  &   \widehat{\varphi}_{ji}
\end{array}
\right\}
$$ 
defines a functor from the
category $I$ to the category of abelian groups (where we set $\widehat{\varphi}_{ii}= {\rm  id}_{\widehat{G}_i}$
for all $i\in I$).
\item[(2)]  For each $i\in I_{free}$ we have that the map
$$\bigoplus _{k<i} \varphi _{ik}\colon  \bigoplus _{k<i}  M_k \to G_i$$
is surjective.
\end{itemize}
\end{enumerate}
We say that an $I$-system $\mathcal  J =  \left(I, \leq , (G_i)_{i\in I}, \varphi_{ji}\, (i<j)\right)$ is {\it finitely generated} 
in case $I$ is a finite poset and all the groups $G_i$ are finitely generated.} 
\end{defi}

\begin{exem}
  \label{exam:Easy-I-system} {\rm We present a family of $I$-systems, where $I$ is a fixed poset.  
  Let $$I= \{p,q_1,q_2,\dots , q_r \},$$ where $p>q_i$ for all $i$, and all $q_i$ are pairwise incomparable.  
We set $I= I_{{\rm free}}$. Since $q_i$ are minimal free primes we must have $G_{q_i}=\{e_{q_i}\}$. To complete the definition of the $I$-system  $\mathcal J$ we only need an
abelian group $G_p$ and semigroup homomorphisms  $\varphi_{p,q_i}\colon \N  \to G_p$ such that $\bigoplus _{i=1}^r \varphi_{p,q_i}\colon \bigoplus_{i=1}^r \N  \to G_p$ is surjective. We distinguish two cases:
\begin{enumerate}
\item If $r=1$, then $G_p$ must be a finite cyclic group. 
\item If $r>1$, then we can take $G_p$ of the form $\Z^s \oplus \Z_{k_1}\oplus \cdots \oplus  \Z_{k_{r-s}} $, where $0\le s< r$, $1\le k_j$ for $1\le j\le r-s$,  
and where $\varphi_{p,q_i}(1)= (0,\dots,0,1,0,\dots ,0)$, with a $1$ in the $i$-th position, for
$1\le i\le s$, and $\varphi_{p,q_i}(1)= (-1,-1,\dots ,-1, \ol{0},\dots , \ol{0},\ol{1},\ol{0},\dots ,\ol{0})$, where there are $s$ $-1$'s, and where $\ol{1}\in \Z_{k_{i-s}}$, for $s<i\le r$. 
\end{enumerate}} 
\end{exem}

To every $I$-system $\mathcal J$ one can associate a primely generated conical refinement monoid $M(\mathcal J)$, and conversely to any primely generated conical refinement monoid $M$, we can associate
an $I$-system $\mathcal J$ such that $M\cong M(\mathcal J)$, see Sections 1 and 2 of \cite{AP} respectively.

Given a poset $I$, and an $I$-system $\mathcal{J}$, we construct a semilattice of groups based on the partial order 
of groups $(I,  \le, \widehat{G}_i)$, by following the model introduced in  \cite{Dobb84}. Let $A(I)$ be
the semilattice (under set-theoretic union) of all the finitely generated lower subsets
of $I$. These are precisely the lower subsets $a$ of $I$ such that the set $\Ma (a)$ of maximal elements of $a$ is finite and each element of $a$ is under some of the maximal ones.
In case $I$ is finite, and since the intersection of lower subsets of $I$ is again a lower subset, $A(I)$ is a lattice. 
For any $a\in A(I)$, we define $\widehat{H}_a= \bigoplus _{i\in a}
\widehat{G}_i$, and we define $f^b_a$ $(a\subseteq b)$ to be the canonical
embedding of $\widehat{H}_a$ into $\widehat{H}_b$. Given $a\in A(I), i\in a$ and $ u\in \widehat{G}_i$, we define $\chi (a,i,u)\in \widehat H _a$ by 
$$\chi (a,i,u)_j=
\left\{
\begin{array}{cc}
  u   & \mbox{if }j=i,  \\
   0_j  &   \mbox{if }j\ne i. 
\end{array}
\right.
$$
Let $U_a$ be the subgroup of $\widehat{H}_a$
generated by the set 
$$\{\chi (a,i,u) -\chi (a,j,\widehat{\varphi}_{ji}(u)) :  i<j\in \Ma (a), u\in \widehat{G}_i\}.$$

Now, for any $a\in A(I)$, set $\widetilde{G}_a= \widehat{H}_a/U_a$, and let $\Phi_a:\widehat{H}_a\rightarrow \widetilde{G}_a$ be the natural onto map. 
Then, for any $a\subseteq b\in A(I)$ we have that $f_a^b(U_a)\subseteq U_b$, so that there exists a unique homomorphism $\widetilde{f}_a^b: \widetilde{G}_a\rightarrow \widetilde{G}_b$ which makes the diagram 
$$
\xymatrix{\widehat{H}_a\ar[r]^{f_a^b} \ar[d]_{\Phi_a} & \widehat{H}_b\ar[d]^{\Phi_b}\\
\widetilde{G}_a \ar[r]_{\widetilde{f}_a^b}  & \widetilde{G}_b}
$$
commutative. Hence,  $(A(I), (\widetilde{G}_a)_{a\in A(I)}, \widetilde{f}_a^b (a\subset b))$ is a semilattice of groups. Thus, the set 
$$\widetilde{M}(\mathcal J ):=\bigsqcup\limits_{a\in A(I)}\widetilde{G}_a,$$ endowed with the operation 
$x+y:= \widetilde{f}_a^{a\cup b}(x)+\widetilde{f}_b^{a\cup b}(y)$ for any $a,b\in A(I)$ and any $x\in \widetilde{G}_a, y\in \widetilde{G}_b$, 
is a primely generated regular refinement monoid by \cite[Proposition 1]{Dobb84}. Note that
$\widehat{H}_{\emptyset}=\widetilde{G}_{\emptyset}=\{ 0 \}$. We refer the reader to \cite{Dobb84} for further details on this
construction.

Let $H_a$ be
the subsemigroup of $\widehat{H}_a$ defined by
$$H_a= \left\{ (z_i)_{i\in a} \in \widehat{H}_a : z_i\in \left\{ 
\begin{array}{ccc}
 \N\times G_i  & \text{ for }  &  i\in \Ma
(a)_{{\rm free}} \\
 \{ (0,0_i) \} \cup (\N\times G_i) & \text{ for }  &  i\in a_{{\rm free}}\setminus \Ma (a)_{{\rm free}} 
\end{array}
\right.
 \right\}.
$$

In what follows, whenever $i<j \in I$ with $j$ a free element, $x= (n,g)\in \N \times G_j$ and $y\in M_i$,
we will see $x+\varphi _{ji}(y)$ as the element $(n, g+\varphi _{ji}(y))\in \N\times G_j$. This is coherent with our identification of $G_j$ as
the subgroup $\{0\} \times G_j$ of $\widehat{G}_j= \Z\times G_j$. 

By \cite[Lemma 1.3]{AP}, we can define a semilattice of semigroups
$$(A(I), (H_a)_{a\in A(I)}, f_a^b (a\subset b)).$$ 
Now, we construct a monoid associated to it. For this, consider the congruence $\sim $ defined on $H_a$, for $a\in A(I)$, given by
$$x\sim y \iff x-y\in U_a .$$

\begin{lem} [{\cite[Lemma 1.4]{AP}}]
\label{lem:equivrels} Let $a\in A(I)$. The congruence $\sim $ on
$H_a$ agrees with the congruence $\equiv$, generated by the pairs $(x+\chi
(a,i,\alpha), x+\chi (a,j, \varphi _{ji}(\alpha) ))$, for $x\in
H_a$,  $i<j\in \Ma (a)$ and $\alpha \in M_i$.
\end{lem}

\begin{corol}[{\cite[Corollary 1.5]{AP}}]
\label{cor:NouSubmonoid}
For every $a\in A(I)$, $M_a:=H_a/{\sim} = H_a/{\equiv} $ is a submonoid of $\widetilde{G}_a$.
\end{corol}

\begin{defi}
\label{def:MJ} {\rm Given an $I$-system $\mathcal J=(I,\le, G_i,\varphi _{ji} (i<j))$, we denote by
$M(\mathcal J)$ the set $\bigsqcup _{a\in A(I)} M_a$. By \cite[Lemma 1.3]{AP} and 
Corollary \ref{cor:NouSubmonoid}, $M(\mathcal J)$  is a submonoid of $\widetilde{M}(\mathcal J)$.}
\end{defi}

Observe that Lemma \ref{lem:equivrels} gives:

\begin{corol} [{\cite[Corollary 1.6]{AP}}]
\label{cor:presentationofM} $M(\mathcal{J})$ is the
monoid generated by $M_i$, $i\in I$, with respect to the defining
relations
$$x+y= x+\varphi _{ji}(y), \quad i<j, \, x\in M_j,\,  y\in M_i.$$
\end{corol}

\begin{nota}{\rm
\label{nota:chisubi} Assume $\mathcal J$ is an $I$-system. For $i\in I$ and $x\in M_i$ we will denote by
$\chi _i(x)$ the element $[\chi (I\dnw i, i, x)]\in M(\mathcal J)$. Note that, by Corollary \ref{cor:presentationofM},
$M(\mathcal J )$ is the monoid generated by $\chi _i (x)$, $i\in I$, $x\in M_i$, with the defining relations
$$\chi _j (x)+\chi _i (y)= \chi _j (x+\varphi _{ji}(y)), \quad i<j, \, x\in M_j,\,  y\in M_i.$$}
\end{nota}

\begin{lem}
 \label{lem:GrotofM} For each $a\in A(I)$, the Grothendieck group of $M_a$ is the group $\widetilde{G}_a$.
 \end{lem}
 \begin{proof}
  Since $M_a$ is a submonoid of $\widetilde{G}_a$, we only have to show that every element of $\widetilde{G}_a$ can be written as
  a difference of two elements of $M_a$. For this, it is enough to show that for each $i\in a$ and each $x\in \widehat{G}_i$, the element 
  $\chi (a, i, x)$ of $\widehat{H}_a$ can be written as a difference of two elements in $H_a$. If $i\in a_{{\rm reg}}$, then we can write
  $$\chi (a, i,x) = \Big( \chi (a,i,x) + \sum_{j\in a_{{\rm free}}} \chi (a, j ,(1,e_j))\Big) - \Big( \sum_{j\in a_{{\rm free}}} \chi (a, j, (1,e_j))\Big) \in H_a - H_a.$$
  If $i\in a_{{\rm free}}$, then select $n\in \N$ such that $(n, e_i)+x\in M_i$, and write
  $$\chi(a,i,x) = $$
  $$\Big( \chi (a, i, (n,e_i)+x) + \sum_{j\in a_{{\rm free}}} \chi (a, j, (1,e_j))\Big) - \Big( \chi (a, i, (n,e_i)) + \sum_{j\in a_{{\rm free}}} \chi (a, j, (1,e_j))\Big) \in H_a-H_a.$$
 \end{proof}

\begin{lem}
 \label{lem:when-a-isidown}
 Let $i\in I$ and set $a= I \dnw i$. Then $M_a= M_i$ and $\widetilde{G}_a = \widehat{G}_i$.
\end{lem}

\begin{proof}
 We have a surjective monoid homomorphism $\phi \colon M_i \to M_a$ sending $x\in M_i$ to $\chi _i (x)$. Define a monoid homomorphism 
 $\psi \colon H_a \to M_i$ by 
 $$\psi (\sum _{j\le i} \chi (x_j, j,a) ) = \sum_{j\le i} \varphi_{ij} (x_j), $$
 where $x_i\in M_i$,  $x_j\in M_j=G_j$ for $j\in a_{{\rm reg}}\setminus \{ i \} $ and $x_j \in \{ (0,e_j)\} \cup M_j$ for $j\in a_{{\rm free}}\setminus \{ i \}$. 
 (We are setting here $ \varphi_{ii}= \text{Id}_{M_i}$.) 
 Then, $\psi $ clearly factors through the congruence $\equiv$ described in Lemma \ref{lem:equivrels}, and so induces a monoid homomorphism $\ol{\psi}\colon M_a\to M_i$, which 
 is clearly the inverse map of
 $\phi$. A similar proof gives that $\widetilde{G}_a = \widehat{G}_i$.
 \end{proof}

We will denote by $\mathcal L (M)$ the lattice of order-ideals of a monoid $M$ and by $\mathcal L (I)$ the lattice of lower subsets
of a poset $I$.

\begin{prop} [{\cite[Proposition  1.9]{AP}}]
 \label{prop:characideals}
 Let $\mathcal J$ be an $I$-system. Then there is a lattice isomorphism $$\mathcal L (I) \cong \mathcal L (M(\mathcal J)).$$
 More precisely, given a lower subset $J$ of $I$, the restricted $J$-system is $$\mathcal J _J :=(J, \le , (G_i)_{i\in J}, \varphi _{ji}, (i<j\in J)), $$
 and the 
 map $J\mapsto M(\mathcal J _J)$ defines a lattice isomorphism from $\mathcal L (I)$ onto $\mathcal L (M(\mathcal J ))$. 
 \end{prop}

 \begin{lem}
 \label{lem:GrotIdeals} Let $I$ be a poset and let $\mathcal J $ be an $I$-system. Let $J$ be a finitely generated lower subset of $I$ and let
 $\mathcal J_J$ be the restricted $J$-system. Then the Grothendieck group of the associated order-ideal $M(\mathcal J_J)$ of $M(\mathcal J)$ is precisely
 $\widetilde{G}_J$. 
  \end{lem}

  \begin{proof}
We have $M(\mathcal J _J) = \bigsqcup _{a\in A(J)} M_a$. Let $x$ be an element in $M_J$, and define a semigroup homomorphism
   $$\tau  \colon M(\mathcal J _J) \longrightarrow G(M_J)$$
   by $\tau (z) = (x+ f_a^J (z))-x$ for $z\in M_a$. Then it is easily seen that $\tau$ is the canonical map from $M(\mathcal J_J)$ to its Grothendieck group, that is, that for every semigroup homomorphism 
   $\lambda \colon M(\mathcal J_J)\to G$, where $G$ is a group, there is a unique group homomorphism $\widetilde{\lambda}\colon G(M_J) \to G$ such that 
   $\lambda = \widetilde{\lambda}\circ \tau$. Indeed, given $\lambda $ as above, $\widetilde{\lambda}$ is just the canonical map from the Grothendieck group
   $G(M_J)$ of $M_J$ to $G$ induced by the restriction of $\lambda $ to $M_J$.

 Now, we can apply Lemma \ref{lem:GrotofM} to derive the result.
\end{proof}

Recall that given a poset $I$, and an element $i\in I$, the {\it lower cover} of $i$ in $I$ is the set 
$$\rL(I,i)=\{j\in I : j<i \text{ and } [j,i]=\{j,i\}\}.$$

We now consider the special case of a finitely generated conical regular refinement monoid $M$. Our goal will be to realize $M$ as the graph monoid
$M(E)$ for some row-finite directed graph $E$. In this case the results from \cite{AP} reduce to Dobbertin's results \cite{Dobb84}.

 For the rest of this subsection, let $M$ be a finitely generated conical \underline{regular} refinement monoid. The antisymmetrization $\ol{M}$ of $M$ is then an antisymmetric 
 regular refinement monoid (cf. \cite[Theorem 5.2]{Brook}). The set $\mathbb P$ is defined by taking a representative $p$ for each prime $\ol{p}\in \mathbb P (\ol{M})$.
 For each $p\in \mathbb P$, the archimedean  component $M_p$ of $p$ is a finitely generated abelian group, denoted by $G_p$. Since we have
 $\ol{p}= \ol{x}$ for all $x\in G_p$, we may take $p=e_p$, the neutral element of the group $G_p$, as a canonical representative of $\ol{p}$, so that 
 $$\mathbb P = \{ e\in M : e=2e \text{ and } e \text{ is prime } \}.$$
 With the order induced from $M$, $\mathbb P$ is a finite poset. Note that $e\le f$ in $\mathbb P$ if and only if $f= e+f$.
 For $e\in \mathbb P$, the associated group is $G_e= \{ x\in M : e\le x \le e \}$, which is precisely the archimedean component $G_M[e]$ of $e$.
 Finally if $e\le f$ in $\mathbb P$, then the induced map $\varphi_{fe}\colon G_e\to G_f$ is defined by $\varphi_{f, e}(x) = x+f$ for $x\in G_e$.
 This structure defines the $\mathbb P$-system $\mathcal J_M$ associated to $M$.

  
\subsection{Graph monoids}

Now, we will recall the basic elements about graphs and their monoids that are necessary in the sequel.

A \emph{(directed) graph} $E=(E^0,E^1,r,s)$ consists of two
countable sets  $E^0,E^1$ and maps $r,s:E^1 \to E^0$. The elements
of $E^0$ are called \emph{vertices} and the elements of $E^1$
\emph{edges}.

A vertex $v\in E^0$ is a sink if $s^{-1}(v)=\emptyset$. A graph $E$
is \emph{finite} if $E^0$ and $E^1$ are finite sets.  If $s^{-1}(v)$
is a  finite set for every $v\in E^0$, then the graph is called
\emph{row-finite}. We will only deal with row-finite graphs in this paper, so 
{\bf we make the convention that all graphs appearing henceforth are row-finite;}
we will make this assumption explicit in the statements of the main results.
A \emph{path} $\mu$ in a graph $E$ is a sequence
of edges $\mu=(\mu_1, \dots, \mu_n)$ such that
$r(\mu_i)=s(\mu_{i+1})$ for $i=1,\dots,n-1$. In such a case,
$s(\mu):=s(\mu_1)$ is the \emph{source} of $\mu$ and
$r(\mu):=r(\mu_n)$ is the \emph{range} of $\mu$. If $s(\mu)=r(\mu)$
and $s(\mu_i)\neq s(\mu_j)$ for every $i\neq j$, then $\mu$  is a
called a \emph{cycle}. We say that a cycle $\mu =(\mu_1, \dots
,\mu_n) $ has an \emph{exit} if there is a vertex $v=s(\mu_i)$ and
an edge $f\in s^{-1}(v)\setminus \{\mu_i\}$. If $v=s(\mu)=r(\mu)$
and $s(\mu_i)\neq v$ for every $i>1$, then $\mu$ is a called a
\emph{closed simple path based at $v$}. For a path $\mu$ we
denote by $\mu^0$ the set of its vertices, i.e.,
$\{s(\mu_1),r(\mu_i)\mid i=1,\dots,n\}$. For $n\ge 2$ we define
$E^n$ to be the set of paths of length $n$, and $E^*=\bigcup_{n\ge
0} E^n$ the set of all paths. 

We define a relation $\ge$ on $E^0$ by
setting $v\ge w$ if there is a path $\mu\in E^*$ with $s(\mu)=v$ and
$r(\mu)=w$. A subset $H$ of $E^0$ is called \emph{hereditary} if
$v\ge w$ and $v\in H$ imply $w\in H$. A set $H$ is \emph{saturated} if
every vertex which feeds into $H$ and only into $H$ is again in $H$,
that is, if $s^{-1}(v)\neq \emptyset$ and $r(s^{-1}(v))\subseteq H$
imply $v\in H$. Denote by $\mathcal{H}$ (or by $\mathcal{H}_E$ when
it is necessary to emphasize the dependence on $E$) the set of
hereditary saturated subsets of $E^0$. 

The set $T(v)=\{w\in E^0\mid v\ge w\}$ is the \emph{tree} of $v$, and it is the smallest hereditary subset of $E^0$  containing $v$. We extend this
definition for an arbitrary set $X\subseteq E^0$ by $T(X)=\bigcup_{x\in X} T(x)$. The \emph{hereditary saturated closure} of a set $X$ is defined as
the smallest hereditary and saturated subset of $E^0$ containing $X$. It is shown in \cite{AMFP} that the hereditary saturated closure of a set
$X$ is $\overline{X}=\bigcup_{n=0}^\infty \Lambda_n(X)$, where
\begin{enumerate}

\item[] $\Lambda_0(X)=T(X)$, and
\item[] $\Lambda_n(X)=\{y\in E^0\mid
s^{-1}(y)\neq \emptyset$ and $r(s^{-1}(y))\subseteq \Lambda_{n-1}(X)\}\cup \Lambda_{n-1}(X)$, for $n\ge 1$.
\end{enumerate}

We recall here some graph-theoretic constructions which will be of interest. For a hereditary subset of $E^0$, the \emph{quotient graph} $E/H$ is
defined as
$$(E^0\setminus H, \{e\in E^1|\ r(e)\not\in H\}, r|_{(E/H)^1}, s|_{(E/H)^1}),$$ 
and the \emph{restriction graph} is
$$E_H=(H, \{e\in E^1|\ s(e)\in H\}, r|_{(E_H)^1}, s|_{(E_H)^1}).$$

\begin{defi}\label{Def:StronglyConnected}
{\rm Given a graph $E$:
\begin{enumerate}
\item We say that $E$ is transitive if every two vertices of $E^0$ are connected through a finite path. 
\item We say that a nonempty subset $S$ of $E^0$ is strongly connected if the graph
$$(S, s^{-1}(S)\cap r^{-1}(S), s\vert_S, r\vert_S)$$
is transitive. In particular, if $F$ is a subgraph of $E$, we say that $F$ is strongly connected if so does the subset $F^0$ of $E^0$.
\end{enumerate}
}
\end{defi}

For a row-finite graph $E$, the {\it graph monoid}  associated to $E$, denoted by $M(E)$, is the commutative monoid given by the generators $\{
a_v\mid v\in E^0\}$, with the relations:
\begin{equation}\label{(M)}
 a_v= \sum _{\{e\in E^1\mid
s(e)=v\}}a_{r(e)}\qquad \text{for every }v\in E^0 \text{ that
emits edges} .
\end{equation}

Let $\mathbb{F}$ be the free commutative monoid on the set $E^0$. The nonzero
elements of $\mathbb{F}$ can be written in a unique form up to permutation
as $\sum _{i=1}^n x_i$, where $x_i\in E^0$. Now we will give a
description of the congruence on $\mathbb{F}$ generated by the relations
(\ref{(M)}) on $\mathbb{F}$. It will be convenient to introduce the
following notation. For $x\in E^0$, write
$${\bf r}(x):=\sum _{\{e\in E^1\mid s(e)=x\}} r(e)\in \mathbb{F} .$$
With this new notation relations (\ref{(M)}) become $x={\bf r}(x)$
for every $x\in E^0$ that emits edges.

\begin{defi}[{\cite[Section 4]{AMFP}}]
\label{binary} {\rm Define a binary relation $\rightarrow_1$ on
$\mathbb{F}\setminus \{0\}$ as follows. Let $\sum _{i=1}^n x_i$ be an
element in $\mathbb{F}$ as above and let $j\in \{1,\dots ,n\}$ be an index
such that $x_j$ emits edges. Then $\sum _{i=1}^n x_i\rightarrow_1
\sum _{i\ne j}x_i+{\bf r}(x_j)$. Let $\rightarrow $ be the
transitive and reflexive closure of $\rightarrow _1$ on
$\mathbb{F}\setminus \{0\}$, that is, $\alpha\rightarrow \beta$ if and only
if there is a finite string $\alpha =\alpha _0\rightarrow _1
\alpha _1\rightarrow _1 \cdots \rightarrow _1 \alpha _t=\beta.$
Let $\sim$ be the congruence on $\mathbb{F}$ generated by the relation
$\rightarrow_1 $ (or, equivalently, by the relation $\rightarrow
$). Namely $\alpha\sim \alpha$ for all $\alpha \in \mathbb{F}$ and, for
$\alpha,\beta \ne 0$, we have $\alpha\sim \beta$ if and only if
there is a finite string $\alpha =\alpha _0,\alpha_1, \dots
,\alpha _n=\beta$, such that, for each $i=0,\dots ,n-1$, either
$\alpha _i\rightarrow _1 \alpha _{i+1}$ or
$\alpha_{i+1}\rightarrow_1 \alpha _i$. The number $n$ above will
be called the {\it length} of the string.}\qed
\end{defi}

It is clear that $\sim $ is the congruence on $\mathbb{F}$ generated by
relations (\ref{(M)}), and so $M(E)=\mathbb{F}/{\sim}$.

\begin{lem}[{\cite[Lemma 4.3]{AMFP}}] 
\label{lem:lem4.3AMFP}
Let $\alpha$ and $\beta$ be nonzero elements in
$\mathbb{F}$. Then $\alpha \sim \beta$ if and only if there is $\gamma\in
\mathbb{F}$ such that $\alpha \rightarrow \gamma$ and $\beta \rightarrow
\gamma$.
\end{lem}

\begin{prop}[{\cite[Proposition 4.4]{AMFP}}] The monoid $M(E)$ associated with any row-finite
graph $E$ is a refinement monoid.
\end{prop}

Let $E$ be a graph. For any subset $H$ of $E^0$, we will denote by $I(H)$ the order-ideal of $M(E)$ generated by $H$.

Recall that we denote by $\mathcal{L}(M)$ the lattice of order-ideals of a commutative monoid $M$. Order-ideals of $M(E)$ correspond to hereditary saturated subsets of $E^0$, as follows:

\begin{prop}[{\cite[Proposition 5.2]{AMFP}}]
\label{prop:order-ideals-Hersats}
For any row-finite graph $E$, there is a natural lattice isomorphism from $\mathcal{H}_E$ to $\mathcal{L}(M(E))$ sending $H\in \mathcal{H}_E$ to the order-ideal $I(H)$ generated by $H$.
\end{prop}

Also, we have:

\begin{lem}[{\cite[Proposition 5.2]{AMFP}, \cite[Lemma 2.1]{APS}}]
Let $H$ be a subset of $E^0$,  with hereditary saturated closure $\overline{H}$. Then $I(H)=I({\overline{H}})$, and
$\overline{H}=I(H)\cap E^0$.
\end{lem}

This means that $I(H)$ is generated as a monoid by the set $\overline{H}=I(H)\cap E^0$. We can improve this result, as follows

\begin{lem}\label{lem:OrderIdealGenerator}
Let $E$ be a graph, and let $H$ be a hereditary subset of $E^0$. Then, the order-ideal $I(H)$ is generated as a monoid by the set $\{a_v : v\in H\}$.
\end{lem}

\begin{proof} Take $v\in \Lambda _1 (H)$. Then $r(s^{-1}(v))$ is a finite subset of $H$, and so
$a_v= \sum _{e\in s^{-1}(v)}  a_{r(e)} $ belongs to the submonoid generated by $\{a_w : w\in H\}$.
By induction, the same happens for a vertex in $\Lambda_n(H)$ for all $n\ge 1$. Since $\ol{H}= \bigcup_{n=0}^{\infty} \Lambda_n (H)$,
and $I(H)$ is generated as a monoid by $\ol{H}$, we obtain the result.
\end{proof}
 
Thus, we can conclude:
 
\begin{lem}\label{lem:Her} 
Let $E$ be a directed graph and let $H$ be a hereditary subset of $E^0$. Then the order-ideal $I(H)$ generated by $H$ is isomorphic to the graph monoid of the restriction graph $E_H$.
\end{lem}
\begin{proof}
Since $H$ is a hereditary subset of $E^0$, the map
$$
\begin{array}{cccc}
\phi: & M(E_H) &\rightarrow   & I(H)  \\
 & a_v & \mapsto  & a_v  
\end{array}
$$
is a well-defined monoid homomorphism. Moreover, $\phi$ is an onto map by Lemma \ref{lem:OrderIdealGenerator}.

To show it is injective, 
let $\alpha $ and $\beta$ be elements in the free commutative monoid generated by $H$, and assume that the elements of 
$M(E)$ represented by $\alpha$ and $\beta $ agree. By Lemma \ref{lem:lem4.3AMFP}, there exists $\gamma \in \mathbb{F}$ such that $\alpha \to \gamma $ and $\beta \to \gamma $. 
By the definition of the relation 
$\to$ and the fact that $H$ is hereditary, it follows that the vertices appearing in $\gamma$ belong to $H$, and that $\alpha \to \gamma $ and $\beta \to \gamma $ in $M(E_H)$. 
Hence, $\alpha$ and $\beta$ represent the same element of $M(E_H)$. 
\end{proof}

Given an order-ideal $S$ of a monoid $M$ we define a congruence
$\sim _S$ on $M$ by setting $a\sim _S b$ if and only if there
exist $e,f\in S$ such that $a+e=b+f$. Let $M/S$ be the factor
monoid obtained from the congruence $\sim _S$; see e.g. \cite{AGOP}.

\begin{lem}[{\cite[Lemma 6.6]{AMFP}}] Let $E$ be a row-finite graph. For a saturated
hereditary subset $H$ of $E^0$, consider the order-ideal
$I(H)$ associated with $H$. Then there is a natural monoid isomorphism $M(E)/I(H)\cong
M(E/H)$.
\end{lem}

\subsection{K-Theory for rings.}

Here, we recall a few elements on K-Theory for Leavitt path algebras, that will be necessary in the sequel.

For a ring $R$ (with local units), let $M_{\infty}(R)$ be the
directed union of $M_n(R)$ ($n\in\mathbb N$), where the transition
maps $M_n(R)\to
M_{n+1}(R)$ are given by $x\mapsto \left( \smallmatrix x&0\\
0&0\endsmallmatrix \right)$. We define $\mathcal V(R)$ to be the set of
isomorphism classes (denoted $[P]$) of finitely generated
projective left $R$-modules, and we endow $\mathcal V(R)$ with the
structure of a commutative monoid by imposing the operation
$$[P]+ [Q] := [P\oplus Q]$$ for any isomorphism classes $[P]$ and
$[Q]$. Equivalently \cite[Chapter 1]{Ros}, $\mathcal V(R)$ can be viewed
as the set of equivalence classes $\mathcal V(e)$ of idempotents $e$ in
$M_\infty(R)$ with the operation
$$\mathcal V(e)+\mathcal V(f) :=
\mathcal V\bigl(  \left( \smallmatrix e&0\\ 0&f
\endsmallmatrix \right) \bigr)$$
for idempotents $e,f\in M_\infty(R)$. The group $K_0(R)$ of a
ring $R$ with local units is the universal group of $\mathcal V(R)$. Recall that, as
any universal group of a commutative monoid, the group $K_0(R)$ has a
standard structure of partially pre-ordered abelian group. The set
of positive elements in $K_0(R)$ is the image of $\mathcal V(R)$ under the
natural monoid homomorphism $\mathcal V(R)\to K_0(R)$.\vspace{.2truecm}

Let $E=(E^0,E^1, r, s)$ be a  graph, and let $K$ be a field. We
define the {\em Leavitt path algebra} $L_K(E)$ associated with $E$
as the $K$-algebra generated by a set $\{v\mid v\in E^0\}$ of
pairwise orthogonal idempotents, together with a set of variables
$\{e,e^*\mid e\in E^1\}$, which satisfy the following relations:

(1) $s(e)e=er(e)=e$ for all $e\in E^1$.

(2) $r(e)e^*=e^*s(e)=e^*$ for all $e\in E^1$.

(3) $e^*e'=\delta _{e,e'}r(e)$ for all $e,e'\in E^1$.

(4) $v=\sum _{\{ e\in E^1\mid s(e)=v \}}ee^*$ for every $v\in E^0$
that emits edges.

Note that the relations above imply that $\{ee^*\mid e\in E^1\}$ is
a set of pairwise orthogonal idempotents in $L_K(E)$. In
general the algebra $L_K(E)$ is not unital, but it has a set of local units given by
$\{ \sum _{v\in F} v  \}$, where $F$ ranges  on all finite subsets of $E^0$. 
So, the above facts apply.

Let $E$ be a graph. For any subset $H$ of $E^0$, with hereditary saturated closure $\overline{H}$, we will denote by $\frak I (H)$ the ideal of $L_K(E)$ generated by $H$.

\begin{theor}[{\cite[Theorem 3.5]{AMFP}}] Let $E$ be a row-finite graph. Then there is a natural
monoid isomorphism $\mathcal V(L_K(E))\cong M(E)$.
\end{theor}

Moreover, by \cite[Theorem 3.5]{AMFP}, \cite[Theorem 5.2]{AMFP}, \cite[Lemma 6.6]{AMFP} and \cite[Lemma 2.3(1)]{APS}, we conclude that

\begin{lem}
\label{lem:quotientsallright}
Let $E$ be a row-finite graph, and let $H$ a hereditary subset of $E^0$. Then, for any field $K$ we have that
$$\mathcal V(L_K(E))/\mathcal V(\frak I (H))\cong  \mathcal V(L_K(E/H))  \cong  \mathcal V(L_K(E)/\frak I (H)).$$
\end{lem}


\section{Representing finitely generated regular refinement monoids}\label{Sect:ReprFinGen}

Dobbertin showed in \cite{Dobb84} that every finitely generated conical regular refinement monoid can be represented as a partial order 
of finitely generated abelian groups (see also \cite{AP}). 
Thus, in order to represent a finitely generated conical regular refinement monoid as a graph monoid, it suffices 
to show that some basic diagrams of abelian groups and homomorphisms of groups can be represented as graph monoids. 
In this section, we will prove that this is possible without further restrictions. In order to make clearer the argument, 
we will divide this task in several steps. \vspace{.2truecm}

First, we will show that any finitely generated abelian group can be represented using a graph monoid.
It is worth to remark here that there are easy constructions of finite graphs realizing the monoids of the 
form $H\cup \{0\}$, where $H$ is a finitely generated abelian group, see for instance \cite[p. 27]{Areal}, where a
construction of the second-named author is outlined.  

We use infinite graphs in Lemma 3.1 because this is definitely needed
to realize a monoid coming from a map between {\it two} finitely generated abelian groups, and we need to prepare 
the ground for that result. The second-named author and Fred Wehrung have shown that in the case where $G=H=\Z_2$ and $\varphi \colon H\to G$ is 
the trivial map sending everything to $0$, the corresponding regular refinement monoid $M(\varphi)$ cannot be realized by using a finite graph. 
However we show in Proposition \ref{L:represmaps} that all the monoids of this sort can be realized using infinite graphs.

\begin{lem}\label{L:represgroup}
Let $H$ be a finitely generated abelian group. Then, $H$ is representable as the semigroup $M(E)^*$ associated to a (not necessarily finite) row-finite directed graph $E$.
\end{lem}
\begin{proof}
We start by fixing notation. Using the Structure Theorem for Finitely Generated Abelian Groups, we can assume that 
$$H=\Z ^r\oplus \Z_{n_1}\oplus \cdots \oplus \Z_{n_s}$$
where $r,s \geq 0$, and $n_i\geq 1$ for all $i$. We define $N:=r+s$.

Now, we will define our graph. The graph $E$ has vertices 
$$E_1^0=\{v_1, \dots , v_{N+1}, v_i^j \, \, (r+1\leq i\leq N, j\geq 1)\} .$$
Instead of fixing what are the edges of $E$ extensively, we will express the relations that these edges define on the graph monoid $M(E_1)$, as follows:
\begin{enumerate}
\item $v_{N+1}=2v_{N+1}+\sum\limits_{i=1}^{r}v_i+\sum\limits_{i=r+1}^{N}n_{i-r}v_i$
\item For every $r+1\leq i\leq N$:
\begin{enumerate}
\item $v_i^1=n_{i-r}v_i+v_i^1+v_i^3$
\item for every $j\geq 1$, $v_i^{2j}=v_i^{2j-1}+v_i^{2j}$
\item for every $j\geq 1$, $v_i^{2j+1}=v_i^{2j}+v_i^{2j+1}+v_i^{2j+3}$
\end{enumerate}
\item For every $1\leq i\leq r$, $v_{i}=2v_{i}+\sum\limits_{j=1, j\ne i}^{r}v_j+\sum\limits_{j=r+1}^{N}n_{j-r}v_j+v_{N+1}$
\item For every $r+1\leq i\leq N$, $v_{i}=\sum\limits_{j=1}^{r}v_j+\sum\limits_{j=r+,1j\ne i}^{N}n_{j-r}v_j+(n_{i-r}+1)v_i+v_{N+1}+v_i^1$
\end{enumerate}

Let us carefully explain which are the properties enjoyed by the graph $E$:
\begin{enumerate}
\item[(a)] First, we will show a graphical representation of the part of $E$ described in point (2) above, for any vertex $v_i$ with $r+1\leq i\leq N$. Consider the subgraph $E_{i}$:\vspace{.2truecm}
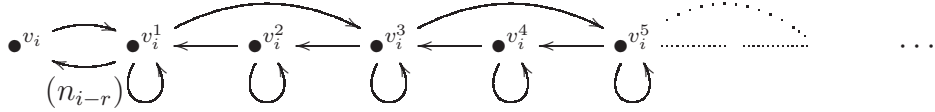
\[
{
\def\labelstyle{\displaystyle}
\xymatrix{
{\bullet}^{v_i} \ar@/^8pt/ [r] & {\bullet}^{v_i^1}\dloopr{} \ar@/^8pt/ [l] ^{(n_{i-r})}  \ar@/^16pt/ [rr] & {\bullet}^{v_i^2}\dloopr{} \ar[l] & {\bullet}^{v_i^3}\dloopr{} \ar[l] \ar@/^16pt/ [rr] & {\bullet}^{v_i^4}\dloopr{} \ar[l] & {\bullet}^{v_i^5}\dloopr{} \ar[l] \ar@/^16pt/@{.} [rr] & \ar@{.}[l] & \ar@{.}[l] & \cdots
}}
\]
Notice that $E_{i}$ is transitive, and moreover every vertex in $E_{i}^0$ is basis point of at least two different simple closed paths. So, the monoid $M(E_{i})$ associated to $E_{i}$ is simple and regular (because in $M(E_{i})$ we have $2v_i\leq v_i$ and $2v_i^j \leq v_i^j$ for every $j\geq 1$).
\item[(b)] By relation (1) above, there are paths connecting $v_{N+1}$ with every vertex $v_i$ ($1\leq i\leq N$), and thus also with every vertex $v_i^j$ ($r+1\leq i\leq N, j\geq 1$) by relations (2) and (4) above. Moreover, in $M(E)$ we have $2v_{N+1}\leq v_{N+1}$.
\item[(c)] By relation (3) above, for every $1\leq k\leq r$, there are paths connecting $v_{k}$ with every vertex $v_i$ ($1\leq i\leq N+1$), and thus also with every vertex $v_i^j$ ($r+1\leq i\leq N, j\geq 1$) by relations (2) and (4) above. Moreover, in $M(E)$ we have $2v_{k}\leq v_{k}$.
\item[(d)] By relation (4) above, for every $r+1\leq k\leq N$, there are paths connecting $v_{k}$ with every vertex $v_i$ ($1\leq i\leq N+1$), and thus also with every vertex $v_i^j$ ($r+1\leq i\leq N, j\geq 1$) by relations (2) and (4) above. Moreover, in $M(E)$ we have $2v_{k}\leq v_{k}$.
\end{enumerate}

As a consequence:
\begin{enumerate}
\item[(e)] The graph $E$ is transitive, and thus $M(E)$ is a simple monoid.
\item[(f)] For every $v\in E_1^0$, $2v\leq v$ in $M(E)$, whence $M(E)$ is regular. Thus, since $M(E)$ is simple, we conclude that $M(E)^*$ is an abelian group.
\end{enumerate}

Now, we will identify this group up to isomorphism. By relation (1) above, the neutral element of $M(E)^*$ is
\begin{equation}\label{eq1}
e=v_{N+1}+\sum\limits_{i=1}^{r}v_i+\sum\limits_{i=r+1}^{N}n_{i-r}v_i.
\end{equation}
In particular, relations (1) and  (3) above simply say that $v_i=v_i+e$ for $i\in \{1, \dots ,r, N+1\}$. Now, by using relation (2) above, we have:
\begin{itemize}
\item $e=n_{i-r}v_i+v_i^3$ by relation (2a).
\item $e=v_i^{2j-1}$ for every $j\geq 1$ by (2b).
\item $e=v_i^{2j}+v_i^{2j+3}$ for every $j\geq 1$ by (2c), and then $e=v_i^{2j}$ for every $j\geq 1$ by the previous identity.
\end{itemize}
As a consequence, $e=v_i^j$ for every $r+1\leq i\leq N$ and every $j\geq 1$, and thus relation (2a) says that 
\begin{equation}\label{eq2}
e=n_{i-r}v_i \text{ for every } r+1\leq i\leq N. 
\end{equation}
Hence, $\langle v_{r+1},\dots ,v_{N}\rangle$ generates a copy of $\Z_{n_1}\oplus \cdots \oplus \Z_{n_s}$ into $M(E)^*$.\vspace{.2truecm}

Replacing equation (\ref{eq2}) in the corresponding places of equation (\ref{eq1}), we obtain
\begin{equation}\label{eq3}
e=v_{N+1}+\sum\limits_{i=1}^{r}v_i.
\end{equation}
Hence, $\langle v_1,\dots v_r, v_{N+1}\rangle$ generates a copy of $\Z^r$ into $M(E)^*$.  Summarizing, the semigroup of nonzero elements of $M(E)$ is isomorphic to $H$, as desired.
\end{proof}

\begin{rema}\label{R:inverses}
{\rm $\mbox{ }$
\begin{itemize}
 \item[(a)] In Lemma \ref{L:represgroup}, we are representing the group $H$ as a group generated by vertices. In order to ease operating with this representation, we need to identify the inverses of the vertices, seen as elements of the group. If we follow the notation of Lemma \ref{L:represgroup}, $\langle v_{r+1},\dots ,v_{N}\rangle$ generates a copy of $\Z_{n_1}\oplus \cdots \oplus \Z_{n_s}$ into $M(E)^*$, and $(n_i-1)v_i$ will be the symmetric of $v_i$ in $M(E)^*$ for $r+1\leq i\leq N$. Also, $\langle v_1,\dots v_r, v_{N+1}\rangle$ generates a copy of $\Z^r$ into $M(E)^*$. In this copy, $\{v_1,\dots, v_r\}$ are the free generators of the group, while $v_{N+1}$ help us to express $r$-tuples of $\Z^r$ with negative entries. To be precise, since $e=v_{N+1}+\sum\limits_{i=1}^{r}v_i$, for any $1\leq i\leq r$ the element 
$$v_1+\cdots +v_{i-1}+v_{i+1}+\cdots +v_r+v_{N+1}$$ 
will be the symmetric of $v_i$ in $M(E)^*$ for $1\leq i\leq r$. 
In order to simplify the notation in the sequel, when we work with a graph monoid $M(E)$, we will denote by $v^{-}$ the symmetric of the vertex $v\in E^0$ seen as an element in the archimedean component of $M(E)$ containing $v$.
\item[(b)] Note that in the isomorphism $M(E)^* \cong H = \Z ^r\oplus \Z_{n_1}\oplus \cdots \oplus \Z_{n_s}$, the vertices $v_1, \dots , v_N$ correspond to the canonical 
generators of $H$ as an abelian group. We will refer to this fact
saying that $v_1,\dots , v_N$ are the canonical generators of $H$. 
\item[(c)] We will work later in Section \ref{TopFree} with {\it semigroup} generators of a group. Observe that $\{ v_1,\dots , v_N, v_{N+1} \}$ is indeed a family of semigroup
generators of $H$. 
\end{itemize}}
\end{rema}

Let $G$ and $H$ be finitely generated abelian groups, and let $\varphi \colon H\to G$ be a group homomorphism. According to \cite{Dobb84}, there exists a regular refinement
monoid $M(\varphi)$ associated to $\varphi$. The monoid 
$M(\varphi)$ has exactly two prime idempotents $e$ and $f$, with $e+f= f$, $G_M[e]= H$, $G_M[f]= G$, and the map 
$\varphi_{f, e}\colon G_M[e]\to G_M[f]$ given by $\varphi _{f, e} (x) = x+f$ is exactly the map $\varphi$. We will show that $M(\varphi)$ is representable as a graph monoid.

\begin{prop}\label{L:represmaps}
Let $\varphi \colon H\to G$ be a homomorphism between finitely generated abelian groups $H$ and $G$, and let $M(\varphi)$ be the associated regular refinement monoid.
Then $M(\varphi)$ is representable as the monoid $M(E)$ of a (not necessarily finite) row-finite directed graph $E$.
\end{prop}

\begin{proof}
We start by fixing notation. Using the Structure Theorem for Finitely Generated Abelian Groups, we can assume that 
$$H=\Z ^r\oplus \Z_{n_1}\oplus \cdots \oplus \Z_{n_s}$$
and
$$G=\Z ^t\oplus \Z_{m_1}\oplus \cdots \oplus \Z_{m_l},$$
where $r,s, t, l \geq 0$, and $n_i,m_j\geq 1$ for all $i,j$. We define $N:=r+s$ and $M:= t+l$. Moreover, if $M<N$, we can add $m_{l+1}=m_{l+2}=\cdots =m_{l+(N-M)}=1$, so that we are thinking 
$$G=\Z ^t\oplus \Z_{m_1}\oplus \cdots \oplus \Z_{m_l}\oplus \Z_1\oplus\cdots \oplus \Z_1.$$
Thus, without loss of generality, we can assume that $N\leq M$. Under this representation of $H$ and $G$, the homomorphism $\varphi$ is represented by a block matrix
$$A=
\left(
\begin{array}{cc}
  A_{1,1} & 0  \\
  A_{2,1}  & A_{2,2}  
\end{array}
\right)
$$
with entries labeled $a_{i,j}\in \Z$, and without loss of generality we can assume that the entries in blocks $A_{2,l}$ ($l=1,2$) satisfy $0< a_{i,j}\leq m_{i-t}$ for every $j$. 

We will define our graph $E$ in two layers, the first one representing the group $H$, while the second will represent simultaneously both the group $G$ and the homomorphism $\varphi$. \vspace{.2truecm}

To construct the first layer, we use Lemma \ref{L:represgroup} to build a graph $E_1$ such that $M(E_1)^*\cong H$, with $$E_1^0= \{v_1, \dots , v_{N+1}, v_i^j (r+1\leq i\leq N, j\geq 1)\},$$ 
such that $v_1,\dots , v_N$ are the canonical generators of $M(E_1)^*=H$ (see Remark \ref{R:inverses}(b)).

Now, we will construct the layer corresponding to $G$, and simultaneously the group homomorphism $\varphi :H\rightarrow G$. To this end, we will define a new collection of 
vertices $E_2^0$, jointly with a family of relations stated in similar terms as those used in Lemma \ref{L:represgroup}. Nevertheless, these new relations will involve not 
only the vertices of $E_2^0$, but also those of $E_1^0$. The new set of vertices $E_2^0$ is 
$$E_2^0=\{w_1, \dots , w_{M+1}, w_i^j (t+1\leq i\leq M, j\geq 1)\} .$$
As in Lemma \ref{L:represgroup}, instead of fixing what are the edges emitted by the vertices of $E_2^0$ extensively, 
we will express the relations that these edges define on the graph monoid. Let $E$ be the graph defined by taking $E^0=E_1^0\sqcup E_2^0$ and where all 
the relations enjoyed by the vertices are as follows:
\begin{itemize}
\item The relations given by the set of edges of the graph $E_1$.
\item The relations we list below:
\begin{enumerate}
\item $w_{M+1}=2w_{M+1}+\sum\limits_{i=1}^{t}w_i+\sum\limits_{i=t+1}^{M}m_{i-t}w_i$
\item For every $t+1\leq i\leq M$:
\begin{enumerate}
\item $w_i^1=m_{i-t}w_i+w_i^1+w_i^3$
\item for every $j\geq 1$, $w_i^{2j}=w_i^{2j-1}+w_i^{2j}$
\item for every $j\geq 1$, $w_i^{2j+1}=w_i^{2j}+w_i^{2j+1}+w_i^{2j+3}$
\end{enumerate}
\item For every $1\leq i\leq t$, 
$$w_{i}=(a_{i,i}+2)w_{i}+\sum\limits_{j=1, j\ne i}^{t}(a_{j,i}+1)w_j+\sum\limits_{j=t+1}^{M}(a_{j,i}+m_{j-t})w_j+w_{M+1}+v_i^{-}.$$
Notice that whenever $a_{j,i}<0$ for some $1\leq j\leq t$, we will replace in the above relation $a_{j,i}w_j$ by $\sum\limits_{k=1,k\ne j}^{t}(-a_{j,i})w_k+(-a_{j,i}w_{M+1})$.
\item For every $t+1\leq i\leq M$, 
$$w_{i}=\sum\limits_{j=1}^{t}(a_{j,i}+1)w_j+\sum\limits_{j=t+1, j\ne i}^{M}(a_{j,i}+m_{j-t})w_j+(a_{i,i}+m_{i-t}+1)w_i+w_{M+1}+w_i^1+v_i^{-}.$$
As in the previous relation, whenever $a_{j,i}<0$ for some $1\leq j\leq t$, we will replace in the above relation $a_{j,i}w_j$ by $\sum\limits_{k=1,k\ne j}^{t}(-a_{j,i})w_k+(-a_{j,i}w_{M+1})$.
\end{enumerate}
\end{itemize}

If $N<M$, and $N<i\le M$, then the term $v_i^-$ appearing in (4) (or in (3) if $N<i \le t$) should be interpreted as the neutral element $0$ of the graph monoid. 
So, for these values of $i$, the relations are of the form (4) (or (3)) in Lemma \ref{L:represgroup}.

Now, consider the graph $E$, and let us identify the properties enjoyed by the elements of $M(E)$:
\begin{itemize}
\item Observe that $E_1^0$ is a hereditary and saturated subset of $E^0$, so that $M(E_1)$ is an order-ideal of $M(E)$, by Lemma \ref{lem:Her}. 
Moreover $M(E_1)^*$ is the archimedean component of $e$
in $M(E)$ and, by construction $M(E_1)^* \cong H$. 
\item By similar arguments to those used in Lemma \ref{L:represgroup}, every vertex in $E_2^0$ is a regular element.
\item By similar arguments to those used in Lemma \ref{L:represgroup}, for any $1\leq i,j\leq M+1$ there exist paths 
from $w_i$ to $w_j$, and also from $w_i$ to $w_j^k$ and from $w_j^k$ to $w_i$ for every $k\geq 1$.
\end{itemize}

Hence, $M(E_1)^*$ equals $\AC{M(E)}{e}$. Also, $E_2^0$ is strongly connected in $E$, and thus the vertices in $E_2^0$ generate an archimedean component $\AC{M(E)}{f}$ of $M(E)$ 
with neutral element $f$. Moreover, $M(E)$ is a regular refinement monoid.

By using the relations defined above, and arguing as in Lemma \ref{L:represgroup}, we have that:
\begin{itemize}
\item By relation (1) above,
\begin{equation}\label{eq4}
f=\sum\limits_{i=1}^{t}w_i+w_{M+1}+\sum\limits_{i=t+1}^{M}m_{i-t}w_i.
\end{equation}
\item By relation (2) above, $f=w_i^j$ for every $t+1\leq i\leq M$ and every $j\geq 1$, while 
\begin{equation}\label{eq5}
f=m_{i-t}w_i \text{ for every } t+1\leq i\leq M.
\end{equation}
\item By replacing equation (\ref{eq5}) in equation (\ref{eq4}), we obtain
\begin{equation}\label{eq6}
f=w_1+\cdots +w_t+w_{M+1}
\end{equation}
\end{itemize}
Hence, the monoid $ \langle w_1,\dots w_M, w_{M+1}\rangle$ generated by $w_1,\dots , w_{N+1}$ is indeed a group, which is a subgroup of $\AC{M(E)}{f}$. 
\vspace{.2truecm}

Now, we will state the relation between $e$ and $f$ in $M(E)$. To this end, notice that:
\begin{itemize}
\item Since $v_1^{-}\in M(E_1)^*$, we have that $e\leq v_1^{-}$.
\item By relation (3) above (or (4) if $t=0$), $v_1^{-}\leq w_1$, so that $e\leq w_1$.
\item By equation (\ref{eq6}) (or (\ref{eq5}) if $t=0$), $w_1\leq f$.
\end{itemize}
Hence, $e\leq f$ in $M(E)$. So there are only two archimedean components in $M(E)$, corresponding to the idempotents $e$ and $f$, and $e\le f$ in $M(E)$.
We want to describe the homomorphism 
$$
\begin{array}{cccc}
\phi_e^f: & \AC{M(E)}{e}  &\rightarrow   &  \AC{M(E)}{f} \\
 & x & \mapsto  & x+f  
\end{array}.
$$
Observing that $v_1,\dots ,v_N$ are canonical generators of
$\AC{M(E)}{e}\cong H$, we are going to compute $\phi_e^f (v_i)$.
We will compute $\phi_e^f(v_i)$ for every $1\leq i\leq N$, by using relations (3) and (4) above. To be precise, recall that $r\leq r+s=N\leq M=t+l$. 
So, it can occur either $N\leq t$ or $N>t$. In the first case, to determine $\phi_e^f(v_i)$ we will only need relation (3) above. 
In the second case, we will use relation (3) above to determine $\phi_e^f(v_i)$ for $1\leq i\leq t$, and relation (4) above to determine $\phi_e^f(v_i)$ for $t+1\leq i\leq N$. 
Let us suppose that we are in the second case. Take $i\in \{1, \dots, t\}$, and add $v_i$ on both sides of relation (3), as follows:
$$w_{i}+ v_i=$$
$$(a_{i,i}+2)w_{i}+\sum\limits_{j=1, j\ne i}^{t}(a_{j,i}+1)w_j+\sum\limits_{j=t+1}^{M}(a_{j,i}+m_{j-t})w_j+w_{M+1}+v_i^{-}+v_i=$$
$$w_{i}+\sum\limits_{j=1}^{M}a_{j,i}w_j+\left[\sum\limits_{j=1}^{t}w_j+w_{M+1}+\sum\limits_{j=t+1}^{M}m_{j-t}w_j\right]+(v_i^{-}+v_i).$$
By definition of $v_i^{-}$ and $w_i^{-}$, we have that $v_i^{-}+v_i=e$ and $w_i^{-}+w_i=f$. Also, $e+f=f$, and by equation (\ref{eq4}), 
$f=\sum\limits_{j=1}^{t}w_j+w_{M+1}+\sum\limits_{j=t+1}^{M}m_{j-t}w_j$. Thus, by adding $w_i^{-}$ on both sides of the previous identity, we have
$$f+v_i= f+\sum\limits_{j=1}^{M}a_{j,i}w_j+f+e=\sum\limits_{j=1}^{M}a_{j,i}w_j+f=\sum\limits_{j=1}^{M}a_{j,i}w_j$$
because $f$ is the neutral element of the group $\AC{M(E)}{f}$. On the other hand, if $i\in \{t+1, \dots, N\}$, a similar argument shows that adding $v_i+w_i^{-}$ on both 
sides of relation (4) gives us
$$v_i+f=\sum\limits_{j=1}^{M}a_{j,i}w_j.$$
Hence, using (\ref{eq5}), (\ref{eq6}) and the above relations, we obtain a monoid homomorphism 
$$\gamma \colon M(\varphi) \to M(E)$$ 
extending the canonical isomorphism 
$H= \AC{M(\varphi)}{e}\to \AC{M(E)}{e}$ and sending the canonical generators of $G= \AC{M(\varphi)}{f}$ to $w_1,\dots ,w_M$. 
We are going to define an inverse $\delta \colon M(E) \to M(\varphi )$ of the map $\gamma $. For this, it is enough to define the map on the vertices of $E$ and to show that the defining
relations of $M(E)$ are preserved by this assignment. The images of the vertices in $E_1$ are dictated by the inverse map of the isomorphism from $H$ onto $\AC{M(E)}{e}$.
Let $\mathbf{x}_1, \dots , \mathbf{x}_M$ be the canonical generators of $G$. We define $\delta (w_i)= \mathbf{x}_i$ for $1\le i \le M$, and 
$$\delta (w_{M+1}) = - (\mathbf{x}_1+\cdots + \mathbf{x}_t).$$
Finally define $\delta (w_i^j) = f$, where $f$ is the neutral element of $G$. It is easy to show that all relations (1)-(4) are preserved by $\delta$, so that 
this assignment gives a well-defined monoid homomorphism $\delta$. It is now clear that $\delta$ is the inverse of $\gamma $. 

Therefore we obtain that $\gamma$ is an isomorphism from $M(\varphi)$ onto $M(E)$. Observe that $\gamma $ sends the canonical set of generators of $G=\AC{M(\varphi)}{f}$
onto $w_1,\dots , w_M$, so that the vertices $w_1,\dots , w_M$ are canonical generators of $\AC{M(E)}{f}$ and the canonical map $\phi^f_e$ has associated matrix $A$ with respect to the canonical
generators $v_1,\dots ,v_N$ of $\AC{M(E)}{e}$ and $w_1,\dots ,w_M$ of $\AC{M(E)}{f}$.
\end{proof}

\begin{rema}
 \label{rem:bigger-than-t}
 {\rm In the above proof, we could have assumed that $N\le l$ and use only relations (4) to encode the map $\varphi$ in the graph monoid $M(E)$.
 The proof is then a little bit shorter, since we would not need to distinguish different cases when we deal with the computation of the
 elements $\phi _e^f (v_i)$. This is the approach that we will follow in the proof of the general case (see Proposition \ref{Th:FinGenRegRefAreRepr2}).} 
 \end{rema}

Let us illustrate our results with two concrete applications of Proposition \ref{L:represmaps}: 

(1) For any $n\in \N$, we will compute the graph $E_n$ such that $M(E_n)$ represents the homomorphism $n\cdot :\Z\rightarrow \Z$. We will follow the same notation as in the proof, so that $r=t=1$, $s=l=0$, $N=M=1$, the matrix $A=(n)$, and for each $n\in\N$ the associated graph $E_n$ is:
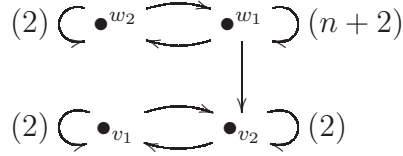
\[
{
\def\labelstyle{\displaystyle}
\xymatrix{  \bullet^{w_2}\dloopd{}_{(2)} \ar@/^8pt/ [r] ^{}& {\bullet}^{w_1}\ar[d]
\ar@/^8pt/ [l] ^{}\uloopd{} ^{(n+2)}\\
 \bullet_{v_1}\dloopd{}_{(2)} \ar@/^8pt/ [r] ^{}& {\bullet}_{v_2}
\ar@/^8pt/ [l] ^{}\uloopd{} ^{(2)}}}
\]
\vspace{.2truecm}

(2) We will compute the graph $F$ such that $M(F)$ represents the group homomorphism $0\cdot :\Z_2\rightarrow \Z_2$. As in the previous example, we will follow the same notation as in the proof, so that $r=t=0$, $s=l=1$, $N=M=1$, the matrix $A=(2)$, and the associated graph $F$ is:
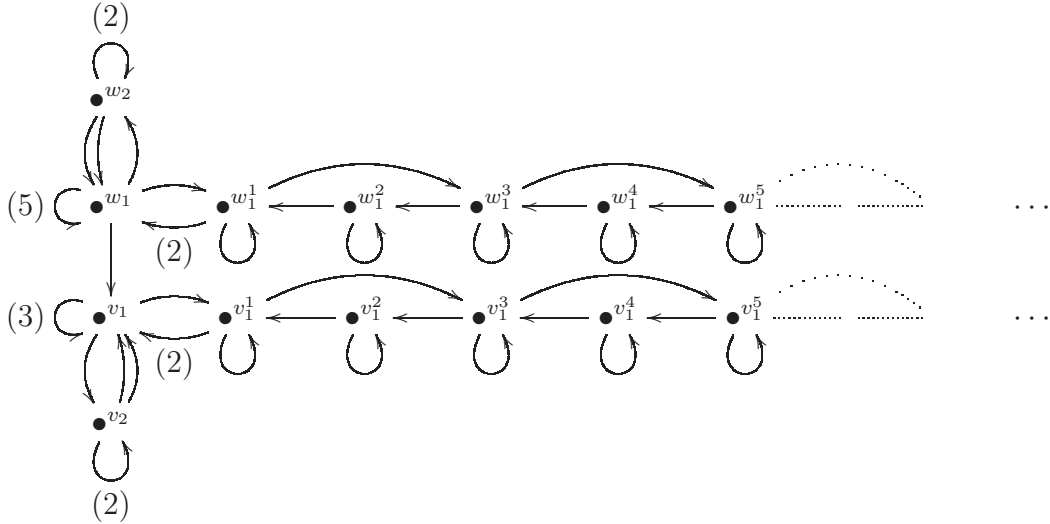
\[
{
\def\labelstyle{\displaystyle}
\xymatrix{
{\bullet}^{w_2} \ar@/^{-10pt}/ [d]\ar@/^{-5pt}/ [d]\uloopr{}^{(2)}&  &  &  & & & &  & \\
{\bullet}^{w_1} \ar@/^{-10pt}/ [u]\ar[d]\dloopd{}_{(5)}\ar@/^8pt/ [r] & {\bullet}^{w_1^1}\dloopr{} \ar@/^8pt/ [l] ^{(2)}  \ar@/^16pt/ [rr] & {\bullet}^{w_1^2}\dloopr{} \ar[l]  & {\bullet}^{w_1^3}\dloopr{} \ar[l] \ar@/^16pt/ [rr] & {\bullet}^{w_1^4}\dloopr{} \ar[l] & {\bullet}^{w_1^5}\dloopr{} \ar[l] \ar@/^16pt/@{.} [rr] & \ar@{.}[l] & \ar@{.}[l] & \cdots\\
{\bullet}^{v_1}  \ar@/^{-10pt}/ [d]\dloopd{} _{(3)}\ar@/^8pt/ [r] & {\bullet}^{v_1^1}\dloopr{} \ar@/^8pt/ [l] ^{(2)}  \ar@/^16pt/ [rr] & {\bullet}^{v_1^2}\dloopr{} \ar[l] & {\bullet}^{v_1^3}\dloopr{} \ar[l] \ar@/^16pt/ [rr] & {\bullet}^{v_1^4}\dloopr{} \ar[l] & {\bullet}^{v_1^5}\dloopr{} \ar[l] \ar@/^16pt/@{.} [rr] & \ar@{.}[l] & \ar@{.}[l] & \cdots\\
{\bullet}^{v_2} \ar@/^{-10pt}/ [u]\ar@/^{-5pt}/ [u]\dloopr{}_{(2)}&  &  &  & & & &  & 
}}
\]
\vspace{.2truecm}

The next step is to show a result analogous to Proposition \ref{L:represmaps}, which allows us to represent confluent maps of groups. This is precisely what we need to use in the inductive step
of the proof of the general result.
The proof is very similar to the one of Proposition \ref{L:represmaps}, so we will only give a brief sketch of it. 

Let $H_1, \dots , H_n, G$ be finitely generated abelian groups, and let $\varphi_i:H_i\rightarrow G$ be group homomorphisms ($1\leq i\leq n$).
Let $M(\varphi)$ be the regular refinement monoid associated to $\varphi := (\varphi_1,\dots ,\varphi_n)$,
that is, $M(\varphi)$ has exactly $n+1$ prime idempotents $e_1,\dots ,e_n$ and $f$, with $e_i+f= f$, and $G_M[e_i]= H_i$, $G_M[f]= G$, and the map 
$\varphi_{fe_i}\colon G_M[e_i]\to G_M[f]$ given by $\varphi _{fe_i} (x) = x+f$ is exactly the map $\varphi_i$ for $1\le i \le n$.

\begin{prop}\label{L:represmaps2}
Let $H_1, \dots , H_n, G$ be finitely generated abelian groups, and let $\varphi_i:H_i\rightarrow G$ be group homomorphisms ($1\leq i\leq n$). Then, these maps can be represented 
simultaneously as the monoid $M(E)$ associated to a (not necessarily finite) row-finite directed graph $E$.
\end{prop}
\begin{proof}
We start by fixing notation. Using the Structure Theorem for Finitely Generated Abelian Groups, we can assume that 
$$H_i=\Z ^{t_i}\oplus \Z_{m_1^i}\oplus \cdots \oplus \Z_{m_{l_i}^i}, $$
for $i=1,\dots , n$, and
$$G=\Z ^r\oplus \Z_{p_1}\oplus \cdots \oplus \Z_{p_s}, $$
where $t_i,l_i, r,s\geq 0$, and $m_i^j, p_i\geq 1$ for all $i,j$. We define  $N_i:=t_i+l_i$ ($i=1,2$) and $M:=r+s$; 
moreover, as in Proposition \ref{L:represmaps}, we can assume  $\sum_{i=1}^n N_i\leq M$.

Now, we will define our graph $E$ essentially as in Proposition \ref{L:represmaps}, defining a first layer corresponding to the groups $H_1,\dots, H_n$, and a second layer corresponding to the group
$G$ and to the different maps $\varphi_i$, which are represented by suitable matrices as in Proposition \ref{L:represmaps}.

To deal with the first layer we use Lemma \ref{L:represgroup} to build $n$ mutually disconnected strongly connected graphs $E_i$ such that $M(E_i)^*\cong H_i$ for $i=1,\dots ,n$, with corresponding sets of vertices 
such that the first $N_i$ vertices of graph $E_i$ are a canonical set of generators of $\AC{M(E_i)}{e_i}= H_i$.

The final step consists of constructing the layer corresponding to the group $G$, together with the homomorphisms $\varphi_i:H_i\rightarrow G$. We will do 
that by adding a new collection of vertices $T^0$, jointly with a family of relations, similar to the ones used in Proposition \ref{L:represmaps}. These relations 
will involve not only vertices in $T^0$, but also in $E_i^0$. The idea is to apply again the procedure described in the proof of Proposition \ref{L:represmaps}. 
Observe that, just as in Proposition \ref{L:represmaps} we obtain an isomorphism $M(\varphi) \to M(E)$ sending the canonical generating sets of each $H_i$ and of $G$ to the canonical generating sets of 
vertices spanning the corresponding archimedean components $\AC{M(E)}{e_i}$ and $\AC{M(E)}{f}$ respectively.   
\end{proof}

We can now obtain the  main result of this section.

\begin{theor}\label{Th:FinGenRegRefAreRepr}
If $M$ is a finitely generated conical regular refinement monoid, then there exists a countable row-finite directed graph $E$ such that $M\cong M(E)$.
\end{theor}
\begin{proof}
By Dobbertin's result \cite{Dobb84}, $M$ is of the form $M (\mathcal J _M )$, for the $I$-system $\mathcal J _M$ of abelian groups $\{ G_M[e]\mid e\in I \}$, 
where $e$ ranges on the poset $I$ of prime idempotents of $M$, as explained in Section \ref{Sect:Basics}. 

Since $M$ is finitely generated, all the groups $G_M[e]$ are finitely generated. We start by fixing a suitable form for these groups, namely we 
set $$G_M[e]= \Z^{r_e} \oplus \Z_{n^e_1}\oplus \cdots \oplus \Z_{n^e_{s_e}}, $$
where $r_e,s_e \geq 0$, and $n^e_i\geq 1$ for all $i$.
Set $N_e = r_e+s_e$. We can assume that for each $f\in I$, we have $\sum _{i=1}^n N_{e_i} \le N_f$, where $\rL(I,f) = \{ e_1,\dots ,e_n\}$ is the lower cover of $f$ in $I$.

We can now proceed to show the result by order-induction. So assume that we have a lower subset $J$ of $I$, and that we have built a countable row-finite graph 
$E_J$ such that $E_J^0 = \bigsqcup _{e\in J} E_e^0$, where each $E_e$ is a strongly connected graph, with 
$$ E_e^0 = \{  v^e_1, \dots , v^e_{N_e+1}, (v^e)_i^j \, \, (r_e+1\leq i\leq N_e, j\geq 1)\} .$$
and an isomorphism
$$\gamma _J\colon M(J)\to M(E_J) ,$$
where $M(J)$ is the order-ideal of $M$ generated by $J$, such that $\gamma_J $ sends the canonical generators of $G_M[e]$ to $v^e_1,\dots , v^e_{N_e}$ for all $e\in J$.  

In case $J\ne I$, let $f$ be a minimal element of $I\setminus J$ and write $J'= J\cup \{f \}$. We will show that the above statement holds for the lower subset $J'$ in place of $J$. This 
clearly establishes the result, because $I$ is a finite poset.

Let $E_f$ be the graph associated to $G_M[f]$, as in Lemma \ref{L:represgroup}.

There are two cases to consider:

(1) \textbf{$f$ is a minimal element of $I$}. Set   $E_{J'}:= E_f\sqcup E_J$. Then 
$$M(J') = M(\{ f \} ) \oplus M(J) \cong M(E_f) \oplus M(E_J) = M(E_{J'})$$
in a canonical way, showing the result.

(2) \textbf{$f$ is not a minimal element of $I$}. Let $\rL(I,f) = \{ e_1,\dots , e_n \}$ be the lower cover of $f$ in $I$. 
Write $\varphi _i = \varphi_{f,e_i}$ for $i=1,\dots , n$, and consider the graph $E$ associated to the maps $\varphi_i$, $i=1,\dots ,n$, as in
Proposition \ref{L:represmaps2}. The first layer of this graph consists exactly of the disjoint union of the graphs $E_{e_i}$, for $i=1,\dots ,n$, so it is a subgraph of our graph $E_J$.
Let $E_{J'}$ be the graph with $E_{J'}^0= E_{J}^0\sqcup E_{f}^0$ and with $E_{J'}^1= E_J^1 \sqcup s_E^{-1}(E_f)$.  
The graph $E_{J'}$ is a countable row-finite graph 
with  $E_{J'}^0 = \bigsqcup _{e\in J'} E_e^0$, where each $E_e^0$ is a strongly connected subset of $E^0$, and the sets of edges $E_e^1$ have the desired form
for all $e\in J'$. We have to prove the existence of the isomorphism $\gamma _{J'}$. 

Since $E_J^0$ is a hereditary and saturated subset of $E_{J'}^0$ we get from Lemma \ref{lem:Her} that the order-ideal of $M(E_{J'})$ generated by
$E_J^0$ is precisely $M(E_J)$. Similarly, the order-ideal of $M(J')$ generated by $J$ is precisely $M(J)$, and $M(J')$ coincides with $M(\mathcal J _{J'})$, which is the 
monoid associated to the partial order of groups $\mathcal J := \mathcal J _M$ restricted to $J'$ (see e.g. \cite[Proposition 1.9]{AP}).
Define a map
$$ \gamma _{J'} \colon M(J') \to M(E_{J'})$$
as follows. The map $\gamma_{J'}$ agrees with the isomorphism $\gamma_J$ when restricted to the order-ideal $M(J)$ of $M(J')$. 
The map $\gamma _{J'}$ restricted to $\AC{M(J')}{f}= \AC{M}{f}$ is just the map obtained by sending the canonical generators of $\AC{M}{f}$ to the canonical generators
$v^f_1, \dots , v^f_{N_f}$ of $\AC{M(E_{J'})}{f}$. In order to prove that $\gamma_{J'}$ gives a well-defined monoid homomorphism, it suffices by \cite[Corollary 1.6]{AP} to show that
if $e< f$ and $x\in \AC{M}{e}$ then $\gamma _{J'}(x) + \gamma _{J'}(f) = \gamma_{J'}(\varphi_{fe}(x))+ \gamma _{J'}(f)$, that is, 
$\gamma _{J'}(x) + f = \gamma_{J'}(\varphi_{fe}(x))$. Since $e<f$ and $I$ is finite, there is some $i$ such that $e\le e_i$. Using the properties of the map defined in 
Proposition \ref{L:represmaps2} and the induction hypothesis, we get
\begin{align*}
 \gamma_{J'}(\varphi _{fe}(x)) & = \gamma_{J'}(\varphi_{f,e_i}(\varphi_{e_i, e}(x)))= \phi_{e_i}^f(\gamma_{J'}(\varphi_{e_i,e}(x)))\\
 & = \phi_{e_i}^f (\gamma _J (\varphi_{e_i,e}(x))) = \phi_{e_i}^f(\phi_e^{e_i} (\gamma _J (x))) \\
 & = \phi^f_e (\gamma_{J'}(x)) = f+ \gamma_{J'}(x).
\end{align*}

This shows that the defining relations of $M(J')$ are preserved and so the map $\gamma_{J'}$ is a well-defined homomorphism from $M(J')$ to $M(E_{J'})$.
To build the inverse $\delta_{J'}$ of $\gamma_{J'}$ we follow the idea in the proof of Proposition \ref{L:represmaps}. The image by $\delta _{J'}$
of the vertices in $E_J$ is determined by the inverse of $\gamma_J$. Let $\mathbf{x}_1, \dots , \mathbf{x}_{N_f}$ be the canonical generators of $\AC{M}{f}$. 
We define $\delta_{J'} (v^f_i)= \mathbf{x}_i$ for $1\le i \le N_f$, and 
$$\delta_{J'} (v^f_{N_f+1}) = - (\mathbf{x}_1+\cdots + \mathbf{x}_{r_f}).$$
Finally define $\delta_{J'} ((v^f)_i^j) = f$. It is easily checked that $\delta_{J'}$ preserves the defining relations of $M(E_{J'})$, and so it gives a well-defined 
homomorphism from $M(E_{J'})$ to $M(J')$, which is clearly the inverse of $\gamma_{J'}$.

This concludes the proof of the result.
\end{proof}


\section{A simple example}\label{Sect:Example}

We are going to represent one of the most simple examples of primitive monoids, using a method similar to the described above. 

This is the monoid
$M=   \Z^+ \cup \{\infty \}$ which appears in \cite{APW08}. Note that this example can be described as $M= \langle p, a \mid a= 2a , a= a+p \rangle $.
Our method is completely different from the method used in \cite{APW08}, and provides a simpler graph than the one given in \cite[Example 6.5]{APW08}.

\begin{lem}
 \label{lem:thefirst}
 The monoid 
 $M= \langle p, a \mid a= 2a , a= a+p \rangle $ can be represented by a graph monoid. 
\end{lem}

\begin{proof}
 The lower component is freely generated by $p$, so we introduce a vertex $v_0$ with a single loop $e_0$ around it. 
 
 Now the second component is regular with trivial associated group, so we use the above method setting $G=\Z _1$, and we consider vertices
 
 $$E^0=\{v_0,  v_1,  v_1^j (j\geq 1)\} .$$
Instead of fixing what are the edges of $E$ extensively, we will express the relations that these edges define on the graph monoid $M(E)$, as follows:
\begin{enumerate}
\item $v_{0}=v_0$
\item 
\begin{enumerate}
\item $v_1^1= v_1+ v_1^1+ v_1^3$
\item for every $j\geq 1$, $v_1^{2j}=v_1^{2j-1}+v_1^{2j}$
\item for every $j\geq 1$, $v_1^{2j+1}=v_1^{2j}+v_1^{2j+1}+v_1^{2j+3}$
\end{enumerate}
\item $v_1 = 2v_1+v_1^1+v_0$.
\end{enumerate}
So, the graph $E$ turns out to be\vspace{.1truecm}

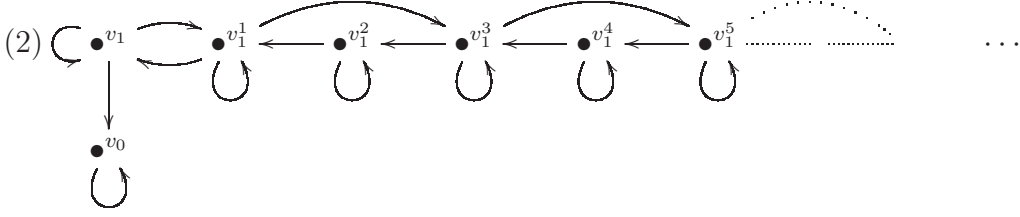
\[
{
\def\labelstyle{\displaystyle}
\xymatrix{
{\bullet}^{v_1} \ar[d]\dloopd{}_{(2)} \ar@/^8pt/ [r] & {\bullet}^{v_1^1}\dloopr{} \ar@/^8pt/ [l]   \ar@/^16pt/ [rr] & {\bullet}^{v_1^2}\dloopr{} \ar[l] & {\bullet}^{v_1^3}\dloopr{} \ar[l] \ar@/^16pt/ [rr] & {\bullet}^{v_1^4}\dloopr{} \ar[l] & {\bullet}^{v_1^5}\dloopr{} \ar[l] \ar@/^16pt/@{.} [rr] & \ar@{.}[l] & \ar@{.}[l] & \cdots\\
{\bullet}^{v_0} \dloopr{}&  &  &  & & & &  & 
}}
\]

 Then one can show as in the previous results that $v_1 =  v_1^j= f$ for all $j\ge 1$, so we obtain from (3) that 
 $f= f+v_0$, so the graph monoid $M(E)$ gives the desired monoid $M$. 
 Note that $M$ is not representable by a graph monoid of a {\it finite} graph by \cite[Theorem 6.1]{APW08}.  
 \end{proof}


\section{Representing finitely generated refinement monoids}\label{TopFree}

In this section, we will obtain our main result (Theorem \ref{thm:main}), which gives a characterization of the finitely generated conical refinement monoids which are graph monoids. 
For this, it is fundamental to use the characterization of these monoids obtained in \cite{AP}, which generalizes the results of Dobbertin \cite{Dobb84} for the regular case, and the results of Pierce 
\cite{Pierce} for the antisymmetric case. \vspace{.2truecm}

First, we will recall some facts from \cite{AP}, that will help to understand the statement of the main result and its proof.

Let $\mathcal{J}=\left(I, \leq , (G_i)_{i\in I}, \varphi_{ji} \, (i<j)\right)$ be an $I$-system as 
in Definition \ref{def:I-system}. For every $i\in I$, $G_i$ is an abelian group. For each $i\in I$, we define a 
commutative semigroup $M_i$ and an abelian group $\widehat{G}_i$ as follows: $M_i=G_i=\widehat{G}_i$ if $i$ is regular, while 
$M_i=\N\times G_i$ and $\widehat{G}_i=\Z\times G_i$ if $i$ is free. Then, the monoid $M(\mathcal{J})$ associated to the $I$-system $\mathcal{J}$ is 
the monoid generated by $\{M_i\}_{i\in I}$, with respect to the defining
relations
$$x+y= x+\varphi _{ji}(y), \quad i<j, \, x\in M_j,\,  y\in M_i.$$

Conversely, given a primely generated refinement monoid $M$, we can associate an $I$-system $\mathcal{J}_M$ to $M$ (\cite[Section 2]{AP}). 
To do this, let $\overline{M}$ be its antisymmetrization, let $\mathbb{P}(\overline{M})$ be its set of prime elements, and let $I \subset \mathbb{P}(M)$ 
be a set of representatives of $\mathbb{P}(\overline{M})$ in $\mathbb{P}(M)$ such that $p=2p$ for any $p\in \mathbb{P}_{reg}$. Then, $I$ is a poset 
with the ordering defined by $q<p$ if and only if $\overline{q}<\overline{p}$ in $ \overline{M}$. The poset $I$ will be called {\it the poset of primes} of $M$. 
For each $p\in I$, let $M_p$ be the archimedean component of $p$ in $M$. 
We have:

\begin{enumerate}
\item If $p$ is regular, then $M_p$ is an abelian group, denoted by $G_p$ (\cite[Lemma 2.7]{Brook}).
\item If $p$ is free, then $G'_p= \{ p+\alpha : \alpha \in M \text{ and } p+\alpha \le p \}$
is an abelian group (with respect to the operation $\circ $ given by $(p+\alpha) \circ (p+\beta ) = p+(\alpha +\beta )$), 
isomorphic to the subgroup $G_p:=\{(p+\alpha)-p : p+\alpha\in G'_p\}$ of $G(M_p)$ (\cite[Remark 2.5]{AP}). Moreover,
$$
\begin{array}{cccc}
\varphi : & \N\times G_p   &\rightarrow   & M_p \\
 & (n, (p+\alpha)-p) & \mapsto  &   np+\alpha \, 
\end{array}
$$
is a monoid isomorphism (\cite[Lemma 2.4]{AP}). 
\end{enumerate}

Thus, given $p,q\in I$ such that $q< p$, the map
$$
\begin{array}{crcl}
\varphi_{pq} : & M_q  &\rightarrow   & G_p  \\
 & x & \mapsto  &   (p+x)-p
\end{array}
$$
is a well-defined semigroup homomorphism, and 
$$\mathcal{J}_M:=\left( I, (G_p)_{p\in I}, \varphi_{pq} (q< p) \right)$$
is the $I$-system associated to $M$.  Moreover, $M(\mathcal{J}_M)$ is naturally isomorphic to $M$ (\cite[Theorem 2.7]{AP}).

\bigskip

Assume that $M$ is a finitely generated conical refinement monoid, with poset of primes $I$. Let $p$ be a 
free prime and let  $\rL(I,p) = \{q_1, \dots , q_n \}$ be the lower cover of $p$.
Then the archimedean component $M_p$ of $p$ has the form $M_{p}= \N\times G_p$ for a finitely generated abelian 
group $G_p$. We will assume that $q_1,\dots , q_r$ are free primes and that $q_{r+1},\dots , q_n$ are regular primes.

We will use here the terminology and notation established in Subsection \ref{subsec:Primely-gen}. 

Let $J_p$ be the lower subset of $I$ generated by $q_1,\dots ,q_n$, and let $M_{J_p}$ be the associated semigroup (cf. Corollary \ref{cor:NouSubmonoid}).
By Lemma \ref{lem:GrotIdeals}, the Grothendieck group of the order-ideal $M(\mathcal J_{J_p})$ associated to $J_p$ is precisely
$\widetilde{G}_{J_p}= G(M_{J_p})$. 

\begin{lem}
 \label{lem:surject}
With the above notation, there exists a surjective semigroup homomorphism 
$$\varphi_p \colon M_{{J_p}} \to G_p$$
induced by the maps $\varphi_{p,q}$ for $q<p$.
 \end{lem}

 \begin{proof}
Let us suppose that $q_1,\dots ,q_r$ are free primes and $q_{r+1},\dots , q_n$ 
are regular primes. Write $S:= \oplus_{q<p} M_q$, and identify $S$ with a subsemigroup of $\widehat{H}_{J_p}= \oplus_{q<p} \widehat{G}_q$. By a slight abuse of notation, we will write $q_i$ for 
$\chi (J_p, q_i,(1,e_{q_i}))\in \widehat{H}_{J_p}$. Note that $H_{J_p}$ is in general a proper subsemigroup of $S$. (See Subsection 2.2 for the definitions of these objects.)
We have a group homomorphism $\psi_p :=\oplus_{q<p} \widehat{\varphi}_{p,q} \colon \widehat{H}_{J_p}\to G_p$ such that $\psi_p(S)=G_p$.
In order to show that $(\psi_p)|_{H_{J_p}}$ is surjective, it is enough to prove that $e_p\in \psi_p(H_{J_p})$. Now, for each $i=1,\dots , r$, we have
$$ - \varphi_{p,q_i} (q_i)= \psi_p (s_i)$$
for some $s_i\in S$. Therefore 
$$e_p= re_p = \psi_p \Big( \sum_{i=1}^r (q_i + s_i) \Big),$$
and $\sum_{i=1}^r (q_i+s_i) \in H_{J_p}$, showing the surjectivity. Now $(\psi_p)|_{H_{J_p}}$ clearly factors through $M_{J_p}= H_{J_p}/{\sim}$, giving rise to a surjective semigroup homomorphism
$\varphi_p\colon M_{J_p}\to G_p$, as desired. 
   \end{proof}

Now, by the universal property of the Grothendieck group, we get a unique group homomorphism
$$G(\varphi_p) \colon \widetilde{G}_{J_p}  \to G_p$$
such that $G( \varphi _p) \circ \iota_{J_p} = \varphi _p$, where $\iota_{J_p} \colon M_{J_p} \to G(M_{J_p})$ is the natural map.

\begin{defi}\label{Def:StriclyPositive}
{\rm
We define the set $G(M_{J_p})^{++}$ of {\it strictly positive} elements of $G(M_{J_p})$ as the image in $G(M_{J_p})$ of the natural map
$\iota_{J_p} \colon M_{J_p}\to G(M_{J_p})$. Note that $G(M_{J_p})^{++}$ is just a subsemigroup of $G(M_{J_p})$, which does not contain the neutral element 
of $G(M_{J_p})$ in general. 
}
\end{defi}

\begin{lem}
 \label{lem:positive-in-G(MJ)} Every element in $\widetilde{G}_{J_p}^{++}= G(M_{J_p})^{++}$ can be represented by an element of the form
 $\sum _{i=1}^r \chi_{q_i}(n_i, g_i) + \sum _{i=r+1}^n \chi _{q_i} (g_i)$ for some $n_i \in \N$, $i=1,\dots , r$  and  $g_i\in G_{q_i}$,
 $i=1,\dots, n$. 
 \end{lem}

 \begin{proof}
  This follows immediately from the description of $M_{J_p}$ given in Section \ref{Sect:Basics}, using that $\{ q_1 ,\dots , q_n \}= \Ma (J_p) $.
   \end{proof}

\begin{exem}
{\rm
 Note that the positive cone $G(M_{J_p})^{++}$ does not coincide in general with the positive cone obtained by considering the image of $M(\mathcal J _{J_p})$ 
 in $G(M(\mathcal J_{J_p}))= G(M_{J_p})$. For instance consider the graph monoid $M= \langle a,b \mid b=b+2a \rangle $, and $J$ the poset of primes $\{ a,b \}$ 
 of $M$. Then $\iota (a)$ is in the image of the canonical map $\iota \colon M\to G(M)$, but $\iota ( a ) \notin G(M_{J_b})^{++}$, since $G(M_{J_b})= \Z\times \Z_2$ 
 and $G(M_{J_b})^{++} = \N\times \Z_2$.
 }
\end{exem}

The following definition is handy to express the conditions charactering finitely generated graph monoids.

\begin{defi}
 \label{def:almost-iso} {\rm Let $G_1$ be an abelian group with a distinguished subsemigroup $G_1^{++}$ of strictly positive elements, and let 
 $G_2$ be an abelian group. We say that a group homomorphism $f\colon G_1\to G_2$ is an {\it almost isomorphism} in case $f$ is surjective and the kernel of $f$
 is a cyclic subgroup of $G_1$ generated by an element in $G_1^{++}$.} 
 \end{defi}

We can now state the main result of the paper.
 
\begin{theor}\label{thm:main}
Let $M$ be a finitely generated conical refinement monoid. Then $M$ is a graph monoid if and only if for each free prime $p$ of $M$, the map
$$ G(\varphi_p) \colon \widetilde{G}_{J_p} \to G_p$$ 
is an almost isomorphism.  
\end{theor}

In particular, we can apply Theorem \ref{thm:main} to obtain, using the results in \cite{AB}, the following  partial affirmative answer to the realization problem for von Neumann regular rings.

\begin{corol}\label{cor:realization}
Let $M$ be a finitely generated conical refinement monoid such that, for all free primes $p$ of $M$, the map $ G(\varphi_p) \colon \widetilde{G}_{J_p} \to G_p$ is an almost isomorphism. 
Then, there exists a (countable) row-finite graph $E$ such that, for any field $K$, the von Neumann regular $K$-algebra $Q_K(E)$ of the quiver $E$ satisfies  $\mathcal V (Q_K(E))\cong M$.
\end{corol}

Theorem \ref{thm:main} enables us to build examples `a la carte' of monoids which are or aren't graph monoids (see also Corollary \ref{cor:APW}). 
The easiest example of a non-graph monoid is still the following:

\begin{exem} {\rm (\cite[Theorem 4.2]{APW08})}
\label{Exam:NoGraphMonoid}
{\rm
Consider the monoid
$$M:=\langle p, a, b : p=p+a,\,  p=p+b\rangle.$$
Then $M$ is a finitely generated conical refinement monoid which is not a graph monoid. 
Indeed, the generators $p, a, b$ are free primes of $M$, and $M$ is an antisymmetric monoid, so that $G_i$ is the trivial group for $i\in \{p,a,b\}$.
The map $G(\varphi_p)\colon \widetilde{G}_{J_p}\to G_p$ becomes the map $\Z^2\to 0$, whose kernel is obviously non-cyclic. 
Therefore $M$ is not a graph monoid by Theorem \ref{thm:main}. 
}
\end{exem}

We will split the proof of Theorem \ref{thm:main} in two parts. First, we prove that the condition is necessary.

\begin{prop}
 \label{thm:free-charac}
 Let $M$ be a finitely generated conical refinement monoid, and let $p$ be a free prime of $M$. If $M$ is a graph monoid then the map
 $$ G(\varphi_p) \colon \widetilde{G}_{J_p} \to G_p$$ is an almost isomorphism.  
\end{prop}
\begin{proof}
Assume that $M$ is a graph monoid. Let $E$ be a (row-finite) countable graph without sinks such that $M\cong M(E)$.
(The condition that $E$ does not have sinks can be assumed because we can add a loop at every sink, without changing the corresponding graph monoid.) 
Let $p$ be a free prime of $M$. Let $\mathcal J = \mathcal J _M$  be the $I$-system associated to $M$, so that $M= M(\mathcal J)$ \cite[Theorem 2.7]{AP}. 
The condition in the statement only depends on the restricted $(I\downarrow p)$-system $\mathcal J_{I\downarrow p}$. Moreover, by 
Proposition \ref{prop:characideals}, Proposition \ref{prop:order-ideals-Hersats} and Lemma \ref{lem:Her}, 
we get
$$M( \mathcal J_{I\downarrow p})\cong I(H) \cong M(E_H) ,  $$ 
where $H$ is the hereditary and saturated subset of $E^0$ corresponding to the order-ideal $M( \mathcal J_{I\downarrow p})$  of $M= M(\mathcal J )$.
Therefore $M( \mathcal J_{I\downarrow p})$ is a graph monoid and, restricting attention to $I\downarrow p$, we may (and will) assume that $p$ is the largest element of $I$. 

We have 
$$ \mathcal V (L_k(E)) \cong M(E) \cong M $$
for any field $k$ \cite[Theorem 3.5]{AMFP}. 
Fix a field $k$ for the rest of the argument. By Proposition \ref{prop:characideals}, 
the order-ideals of $M$ correspond to the lower subsets of the poset $I$.
Moreover, the order-ideals of $M$ correspond to the graded-ideals of $L_k(E)$ and to the hereditary saturated subsets of $E^0$ \cite[Theorem 5.3]{AMFP}. 

Since $p$ is the largest element of $I$, and $p+x= p + \varphi_{p,q} (x)$ for all $x\in M_q$, we get that $K_0(L_k(E))= G(M)= G(M_p)= \Z \times G_p$, which is denoted by 
$\widehat{G}_p$ (see Lemmas \ref{lem:when-a-isidown} and  \ref{lem:GrotIdeals}).
Now, let $\frak I$ be the graded ideal of $L_k(E)$ corresponding to the lower subset $J_p$ of $I$ generated by the primes $q_1,\dots , q_n$ in the lower cover of $p$. Let $H$ 
be the hereditary and saturated subset
of $E^0$ corresponding to $\frak I$, so that $\frak I = \frak I (H)$. Since $J_p$ is a maximal lower subset of $I$, it follows that $\frak I= \frak I (H)$ is a maximal graded-ideal of $L_k(E)$, and 
at the same time it gives rise to the maximal order-ideal $S:= M(\mathcal J _{J_p})$
of $M$ associated to the restricted $J_p$-system $\mathcal J _{J_p}$. Observe that $M/S \cong \Z^+$. 
Moreover, by Lemma \ref{lem:quotientsallright}, we have
$$\mathcal V(L_k (E)/\frak I) \cong \mathcal V(L_k(E))/\mathcal V(\frak I) \cong M/S \cong \Z^+ .$$
Since  $L_k (E)/ \frak I = L_k (E)/ \frak I (H)\cong L_k (E/H)$ is a graded-simple Leavitt path algebra, the tricotomy holds for $L_k(E/H)$
(see \cite[Proposition 3.1.14]{AAS}). Since $\mathcal V(L_k(E/H))\cong \Z^+$,
$L_k(E/H)$ must be Morita-equivalent to either $k$ or $k[t,t^{-1}]$, the $k$-algebra of Laurent polynomials. As the graph $E$ is row-finite and does not have 
sinks, the only possibility is that $L_k(E/H)$ is Morita-equivalent to $k[t,t^{-1}]$. Thus, the graph $E/H$ must have a unique cycle $c$, without exits, to which all the other vertices 
of $E/H$ connect. 
Let $E'$ be the graph obtained from $E$ by removing all vertices in $E^0\setminus (H\cup c^0)$ and all the edges emitted by them. Then
$(E')^0$ is hereditary in $E^0$, and its saturation in $E$ is precisely $E^0$. Therefore $ M(E')=M(E) $ by Lemma \ref{lem:Her}. Hence, replacing $E$ with $E'$, we can assume that 
$E/H$ consists exactly of a unique cycle $c$. Further, let $v$ be a vertex in the cycle $c$. It is easily shown, by a direct computation which only involves
the definition of the graph monoid, that replacing the cycle $c$ with a single loop based at $v$ 
does not change the monoid $M(E)$. (In this construction, we just re-define the starting vertex of the edges emitted by vertices of the cycle to $H$ to be $v$.)
Thus,  we can assume that $E/H$ is a single loop $c$, based at $v$. In particular, we have
$$L_k(E)/ \frak I (H) \cong L_k(E/H) \cong k[t,t^{-1}].$$
Observe that $K_0(\frak I) = G (\mathcal V(\frak I))= G (M_{J_p}) = \widetilde{G}_{J_p}$ by Lemma \ref{lem:GrotIdeals}. 
Hence, the map $K_0(\frak I) \to K_0(L_k(E))$ can be identified with the map $\iota \circ G(\varphi_p)$, where $\iota \colon G_p \to \widehat{G}_p= \Z \times G_p$ is the canonical inclusion. 
On the other hand, we have
$$K_1(L_k(E)/\frak I)= K_1(k [t,t^{-1}])= K_1(k) \oplus K_0 (k) \cong k^{\times}\oplus \Z$$ 
(see e.g. \cite[Theorem III.3.8]{Weibel}). Clearly, we have that the factor $k^{\times}$, which is generated by the units of the field $k$, 
is contained in the image of the map $K_1(L_k(E))\to K_1(L_k(E)/\frak I)$, while the factor $\Z$ is generated by multiplication by the unit $t$ of $k[t,t^{-1}]$. 

Hence, the exact sequence in $K_0$ and $K_1$ corresponding to the short exact sequence
\begin{equation*}
 \begin{CD}
0 @>>> \frak I  @>>>  L_k(E) @>>> k[t,t^{-1}] @>>>  0   
 \end{CD}
\end{equation*}
gives the exact sequence
\begin{equation}
 \begin{CD}
\Z  @>{\partial }>>  \widetilde{G}_{J_p} @>{\iota \circ G( \varphi_p )}>> \widehat{G_p} @>>> \Z @>>>  0 .  
 \end{CD}
\end{equation}
The map $\partial \colon \Z \to \widetilde{G}_{J_p}$ is induced by the connecting homomorphism in algebraic K-theory.
Since we can lift the unit $t$ to the von Neumann regular element $c$ in $v L_k(E) v$, we obtain (cf. \cite[Proposition 1.3]{MM})  that 
$$ \partial ([t]) = [v-cc^*] - [v-c^*c] \in K_0 (\frak I ) .$$
Now, $c^*c= v$ and $cc^*+ \sum _{ e\in s^{-1}(v)\setminus \{c \} } ee^*= v$, so that
$$\partial ([t]) =   \sum _{ e\in s^{-1}(v)\setminus \{c \} } [ee^*] =  \sum _{ e\in s^{-1}(v)\setminus \{c \} } [r(e)] \in K_0 (\frak I).$$
Hence, the kernel of the canonical map $G(\varphi_p )\colon \widetilde{G}_{J_p}   \to  G_p$ is generated by the element
in $\widetilde{G}_{J_p}= G(M_{J_p})$ corresponding to the image of the element  $\sum _{ e\in s^{-1}(v)\setminus \{c \} } [r(e)]$ under the isomorphism
$K_0(\frak I) \to G(M_{J_p})$. 

It remains to show that  $\sum _{ e\in s^{-1}(v)\setminus \{c \} } [r(e)]$ is a strictly positive element of $\widetilde{G}_{J_p}$.
Let $\gamma \colon M \to M(E)$ be the isomorphism between $M$ and $M(E)$. Then each $\gamma (q_i)$ is a prime element of the graph monoid $M(E_H)$
(see Lemma \ref{lem:Her}). Consider the tree $T(v)$ of $v$. Then, $T(v)$ is a hereditary subset of $E^0$ containing $v$, and by hypothesis 
the order-ideal $\frak I (T(v))$ generated by $T(v)$ must be $M(E)$. 
By Lemma \ref{lem:OrderIdealGenerator}, $\frak I (T(v))$ (and so $M(E)$)  is generated {\it as a monoid}  by all the elements of the form $a_w$ with $w\in T(v)$. 
Fix an index $i\in \{ 1,\dots , n \}$. By the preceding argument,  we can write  
$\gamma (q_i) = \sum _{j=1}^m a_{w_j}$, where $w_j\in T(v)$ for all $j$.
Since $a_{w_j}\leq \gamma (q_i)$ for all $j$, we see that all $w_j\in H$.
Since $\gamma (q_i) $ is prime, we get that $\gamma (q_i) \le a_{w_j}$ for some $j$. So, 
for some vertex $w\in T(v)\cap H$, $\gamma (q_i) \equiv a_w$ (where $\equiv$ denotes the antisymmetric relation generated by $\leq$). Thus, 
we can choose an edge $e\in s^{-1}(v)\setminus \{ c \}$ such that $r(e)$ connects to $w$, and so $\gamma (q_i) \le a_{r(e)}$. This shows that
 $q_i \le \gamma ^{-1} (a_{r(e)})\in M(\mathcal J _{J_p})$. Therefore, there exists some lower subset $J'$ of ${J_p}$ such that $\gamma ^{-1} (a_{r(e)})\in M_{J'}$, and since 
 $q_i \le \gamma ^{-1} (a_{r(e)})$, it follows that $q_i \in J'$. Since this holds for every $i=1, \dots , n$, we conclude that $\gamma ^{-1} (\sum _{ e\in s^{-1}(v)\setminus \{c \} } a_{r(e)})$
 belongs to the maximal component $M_{J_p}$ of $M(\mathcal J_{J_p})$. Hence, $\sum _{ e\in s^{-1}(v)\setminus \{c \} } [r(e)]$ is the image of 
 $\gamma ^{-1} (\sum _{ e\in s^{-1}(v)\setminus \{c \} } a_{r(e)})$ under the canonical map $M_{J_p} \to G(M_{J_p})$, and so it is strictly positive in $G(M_{J_p})$.  
 
 This shows that $G(\varphi _p)$ is an almost isomorphism.
\end{proof}

To show the converse, we will use a method which is similar to the method employed in the proof of Theorem \ref{Th:FinGenRegRefAreRepr}. 

\begin{point}\label{Point:Canonical}
{\rm
In order to establish the correct setting for the induction argument, we need to introduce some terminology and notation. 

Since $M$ is finitely generated, all the groups $G_M[p]$, for $p$ a regular prime, are finitely generated. In the case where $p$ is regular, we will assume that
$p$ is the neutral element of the group $G_p = M_p=G_M[p]$.

If $p$ is a free prime, then the archimedean component $G_M[p]$ is of the form 
$$G_M[p]= \N \times G_p, $$
where $G_p$ is a finitely generated group. Here the free prime $p$ is identified with the element $(1, e_p)$ of $\N\times G_p$. 

We start by fixing a suitable form for these groups, namely we 
set, for $p=e$ a regular prime,
$$G_e= \Z^{r_e} \oplus \Z_{n^e_1}\oplus \cdots \oplus \Z_{n^e_{s_e}}, $$
where $r_e,s_e \geq 0$, and $n^e_i\geq 1$ for all $i$.  For $p=e$ a regular prime, set $N_e = r_e+s_e$. 
We denote by $\mathbf{x}_1, \dots , \mathbf{x}_{N_e}$ the canonical set of group generators of the group 
$G_e= \Z^{r_e} \oplus \Z_{n^e_1}\oplus \cdots \oplus \Z_{n^e_{s_e}}$. Further, we set $\mathbf{x}_{N_e+1}= -(\mathbf{x}_1+\dots + \mathbf{x}_{r_e})$, and we observe that
$\mathbf{x}_1, \dots , \mathbf{x}_{Ne}, \mathbf{x}_{N_e+1}$ is a family of semigroup generators for $G_e$, which we will call the {\it canonical family of semigroup generators for}
$G_e$.

For $p$ a free prime, we will denote by $ \{ g^p_1,\dots , g^p_{N_p} \}$ a family of {\it semigroup generators} of the group $G_p$, that is,
every element of $G_p$ is a finite sum of some of the elements in the family $ \{ g^p_1,\dots , g^p_{N_p} \}$. Now, we have the following result:
}
\end{point}

\begin{lem}\label{lem:Auxiliar1}
If $p$ is a free prime, then we can assume that each $g^p_i$ is of the form $\varphi _{p, q} (g)$ for $q<p$, where $g$ is either one of the canonical semigroup generators of $G_q$ if $q$ is a regular prime, or $g$ is $q$ if $q$ is a free prime.
\end{lem}
\begin{proof}
We will show the result  by (order-)induction. If $p$ is a minimal prime which is free, then $G_p$ is a trivial group, and the statement holds 
vacuously, taking the empty family of generators for $G_p$. Assume now that $p$ is a free prime, and that the statement holds for all free primes below $p$.
By \cite[Definition 1.1(c2)]{AP} the map 
$$\bigoplus_{q<p} \varphi_{p,q} \colon \bigoplus _{q < p} M_q\to G_p$$
is surjective. Thus, a family of semigroup generators for $G_p$ is obtained by taking $\varphi _{p,\tilde{q}} (h)$, where $\tilde{q} <p$ ranges on the set of regular primes below $p$, and $h$ ranges on the family 
of canonical semigroup generators
of the group $G_{\tilde{q}}$, together with the family $\{ \varphi _{p,q}(q) \} \cup \{ \widehat{\varphi}_{p,q} ( h) \}$, where $q<p$ ranges on the set of free primes below $p$, and $h$ ranges on the family 
of canonical semigroup generators of $G_q$. 
By induction hypothesis, applied to the free prime $q$, each $h$ can be taken of the form $\varphi _{q,q'} (h')$, where either $q' <q$ is a regular prime and $h'$ is a canonical semigroup generator in $G_{q'}$, 
or it is the form $\varphi _{q,q'} (q')$, where $q' <q$ is 
a free prime. In the former case, we get
$$\widehat{\varphi}_{p,q} (h) = \widehat{\varphi}_{p,q} (\varphi _{q,q'} (h'))= \widehat{\varphi}_{p,q} (\widehat{\varphi} _{q, q' } (h'))= \widehat{\varphi}_{p, q' } (h')= \varphi _{p, q'}(h') ,$$ 
and in the latter case we get
$$\widehat{\varphi}_{p,q} (h)= \widehat{\varphi} _{p,q} (\widehat{\varphi}_{q,q'}(q')) = \widehat{\varphi}_{p,q'} (q') =  \varphi _{p,q'} (q'),$$
which shows the result. 
\end{proof}

\begin{defi}\label{Def:StandardForm}
{\rm 
Let $p\in I_{{\rm free}}$. We say that a family of elements  $\{ g^p_1,\dots , g^p_{N_p} \}$ of $G_p$  is {\it the canonical set of semigroup generators}  if this family consists of all elements 
of the form $\varphi _{p, q} (g)$ for $q<p$, where $g$ is either one of the canonical semigroup generators of $G_q$ if $q$ is a regular prime, or $g$ is $q$ if $q$ is a free prime.
}
\end{defi}

We are now ready to prove the converse of Proposition \ref{thm:free-charac}.

\begin{prop}\label{Th:FinGenRegRefAreRepr2}
Let $M$ be a finitely generated conical refinement monoid such that the natural map $G(\varphi_p )\colon \widetilde{G}_{J_p} \to G_p$ is an almost isomorphism
for every free prime. Then there exists a countable row-finite directed graph $E$ such that $M\cong M(E)$.
\end{prop}
\begin{proof}
Notice that, thanks to Lemma \ref{lem:Auxiliar1}, we can always choose a canonical set of semigroup generators of $G_p$ for any free prime $p\in M$. 
We will also require that $r + \sum _{i=1}^n N_{q_i} \les s_p$ for each $p\in I_{{\rm reg}}$, where $\rL (I, p) = \{ q_1,\dots , q_n \}$ is the lower cover of $p$, 
and $r$ is the number of free primes in $\rL (I,P)$. 

We will show the result by order-induction. 

Assume that we have a lower subset $J$ of $I$, and that we have built a countable row-finite graph 
$E_J$ such that $E_J^0 = \bigsqcup _{q\in J} E_q^0$, where each $E_q$ is a strongly connected graph, with 
$$ E_q^0 = \{  v^q_1, \dots , v^q_{N_q+1}, (v^q)_i^j \, \, (r_q+1\leq i\leq N_q, j\geq 1)\} $$
for a regular prime $q\in J$, and $E_q^0= \{v^q \}$ if $q\in J$ is a free prime.  
We also assume that there is an isomorphism
$$\gamma _J\colon M(J)\to M(E_J) ,$$
where $M(J)$ is the order-ideal of $M$ generated by $J$, such that $\gamma_J $ sends the canonical semigroup generators $\mathbf{x}^q _1,\dots , \mathbf{x}^q_{N_q+1}$ of $G_q$ 
to $v^q_1,\dots , v^q_{N_q+1}$ for all regular $q\in J$, and sends the
element $q$ of $G_M[q]$ to $v^q$ for all free  $q\in J$.

For $q'<q$ in $J$ and $x\in M_{q'}$, we have $\gamma_J (\varphi _{q,q'}(x )) = \phi^{q}_{q'}(\gamma_J (x))$, where $\phi_{q'}^q \colon M(E_J)[q'] \to G(E_J)_q$ is the structural map associated to the 
$\mathcal J _J$-system coming from the finitely generated conical refinement monoid $M(E_J)$. 

In case $J\ne I$, let $p$ be a minimal element of $I\setminus J$ and write $J'= J\cup \{p \}$. We will show that the above statement holds for the lower subset $J'$ in place of $J$. This 
clearly establishes the result, because $I$ is a finite poset.\vspace{.2truecm}

There are two cases to consider:

\noindent (1) \textbf{$p$ is a minimal element of $I$}: In this case, we will associate a graph $E_p$ to $G_M[p]$:  when $p$ is a 
regular prime, we define $E_p$ using Lemma \ref{L:represgroup}, while in case $p$ is a free prime, we take $E_p$ to be the one-loop graph based at the vertex $v^p$. Now,  let 
$E_{J'}:= E_p\sqcup E_J$. Then 
$$M(J') = M(\{ p \} ) \oplus M(J) \cong M(E_p) \oplus M(E_J) = M(E_{J'})$$
in a canonical way, showing the result.\vspace{.2truecm}

\noindent (2) \textbf{$p$ is not a minimal element of $I$}: In this case, let $\rL(I,p) = \{ q_1,\dots , q_n \}$ be the lower cover of $p$ in $I$. Here, we will assume that
$q_1, \dots , q_r$ are free primes and $q_{r+1},\dots , q_n$ are regular primes.  We will define $E_{J'}^0= E_J^0\sqcup E_p^0$, where $E_p^0$ will be specified later. The edges 
in the graph $E_{J'}$ will be the edges coming from $E_J$ and a new family of edges that we will describe. Now, we have two different cases to consider.

(i) \underline{$p$ is a regular prime}: 
In this case, the set of vertices $E_p^0$ is defined as in Lemma \ref{L:represgroup}, so that
$$ E_p^0 = \{  v^p_1, \dots , v^p_{N_p+1}, (v^p)_i^j \, \, (r_p+1\leq i\leq N_p, j\geq 1)\} .$$
The relations $\mathcal{R}$ that define the {\it new} edges of the graph $E_{J'}$, all departing from the vertices in $E_p^0$, are as follows:
\begin{enumerate}
\item $v^p_{N_p+1}=2v^p_{N_p+1}+\sum\limits_{i=1}^{r}v^p_i+\sum\limits_{i=r+1}^{N_p}n^p_{i-r}v^p_i$
\item For every $r+1\leq i\leq N$:
\begin{enumerate}
\item $(v^p)_i^1=n^p_{i-r}(v^p)_i+(v^p)_i^1+(v^p)_i^3$,
\item for every $j\geq 1$, $(v^p)_i^{2j}=(v^p)_i^{2j-1}+(v^p)_i^{2j}$,
\item for every $j\geq 1$, $(v^p)_i^{2j+1}=(v^p)_i^{2j}+(v^p)_i^{2j+1}+(v^p)_i^{2j+3}$,
\end{enumerate}
\item For every $1\leq i\leq r$, $$v^p_{i}=2v^p_{i}+\sum\limits_{j=1, j\ne i}^{r}v^p_j+\sum\limits_{j=r+1}^{N_p}n^p_{j-r}v^p_j+v^p_{N_p+1} \, \, ,$$
\item For every $r+1\leq i\leq N_p$, $$v^p_{i}=\sum\limits_{j=1}^{r}v^p_j+\sum\limits_{j=r+,1j\ne i}^{N_p}n^p_{j-r}v^p_j+(n^p_{i-r}+1)v^p_i+v^p_{N_p+1}+(v^p)_i^1+ A_i+ v(i) \, . $$
\end{enumerate}
Here, for each $r+1\leq i\leq N_p$:
\begin{itemize}
\item $A_i$ is a nonnegative integral linear combination of the vertices $v^p_1,\dots, v^p_{N_p+1}$, which depends on an element $g(i)\in \bigsqcup_{i=1}^n \widehat{G}_{q_i}$, as described below. 
\item $v(i)$ is a certain vertex in the graph $E_J$, which will be described below. 
\end{itemize}
There is a map $i\mapsto g (i)$ from $[r_p+1, r_p + r + \sum_{i=1}^n N_{q_i}]\cap \Z$ to $\bigsqcup_{i=1}^n \widehat{G}_{q_i}$
such that  
\begin{itemize}
\item sends the set $[r_p+1, r_p + r]\cap \Z$ bijectively to the set of free primes $q_1,\dots , q_r $ in $\rL (I,p)$. 
\item establishes a correspondence between $[r_p+r+ 1, r_p + r + \sum_{i=1}^r N_{q_i}]\cap \Z$ 
and $\bigcup _{i=1}^r \{ g^{q_i}_j: j=1,\dots , N_{q_i} \}$, and
\item establishes a correspondence between $[r_p + r+ \sum_{i=1}^r N_{q_i} +1,\,  r_p + r+ \sum _{i=1}^n N_{q_i}]\cap \Z$ and 
$\bigcup _{t=r+1}^{n} \{ {\bf x}^{q_t}_j: j=1,\dots , N_{q_t} \}$. 
\end{itemize}
Of course, we take $A_i$ as the trivial linear combination (with all coefficients being $0$)
in case $i$ is larger than $r_p + r + \sum _{i=1}^n N_{q_i} $.  

Now, we specify the value of the term $A_i$ and the corresponding vertex $v(i)$, which depend on the form of the specific generator $g(i)$. 
Suppose first that $i$ belongs to the interval $[r_p + r+ \sum_{i=1}^r N_{q_i} +1,\,  r_p + r+ \sum _{i=1}^n N_{q_i}]$. In this case $g(i)$ is a canonical {\it group generator}
of a group $G_q$ corresponding to a regular prime $q$ in the lower cover of $p$. We write
$$- \varphi _{p,q}(g(i)) = \sum _{j=1}^{N_p+1} a_{ji} \mathbf{x}^p_j $$
for some non-negative integers $a_{ji}$. 
Then, we define
$$A_i = \sum _{j=1}^{N_p+1} a_{ji} v^p_j, \, \text{ and } \,   v(i) = \gamma _J (g(i)) .$$
We next consider the case where $i$ belongs to the interval $[r_p+1, r_p+r]$, so that $g(i)=q$ for a free prime $q$ in the lower cover of $p$. 
Then we write
$$- \varphi _{p,q}(q) = \sum _{j=1}^{N_p+1} a_{ji} \mathbf{x}^p_j $$
for some non-negative integers $a_{ji}$, and we define 
$$A_i = \sum _{j=1}^{N_p+1} a_{ji} v^p_j , \, \text{ and } \,    v(i) = v^q= \gamma _J (q) .$$
Finally, we need to consider the case where $g{(i)}$ is a canonical semigroup generator of the group $G_q$ for a free prime $q$ in the lower cover of $p$. 
This means by definition that either $g{(i)}$ is of the form $\varphi_{q,q'}(h)$, where
$q'<q$ is a regular prime and $h$ is a canonical semigroup generator of $G_{q'}$, or that  $q'<q$ is a free prime and $g{(i)}= \varphi _{q,q'}(q')$.
In the former case, we set
$$- \varphi _{p, q'} (h)=  \sum _{j=1}^{N_p+1} a_{ji} \mathbf{x}^p_j $$
for some non-negative integers $a_{ji}$, and we define
$$A_i = \sum _{j=1}^{N_p+1} a_{ji} v^p_j , \, \text{ and } \,\,    v(i) =  \gamma _J (h) .$$
In the latter case, we compute the non-negative integers $a_{ji}$ using $-\varphi _{p,q'} (q')$, and we define
$$A_i = \sum _{j=1}^{N_p+1} a_{ji} v^p_j ,  \, \text{ and }  \,\,   v(i) =  v^{q'} = \gamma _J (q') .$$

Note that, in this situation, the same arguments as in Proposition \ref{L:represmaps} give that the subgraph $E_p$  is strongly connected, and that there is a group homomorphism 
$\gamma _p \colon G_p\to M(E_{J'})[f]$ from the group $G_p$ to the archimedean
component $M(E_{J'})[f]$ of the graph monoid $ M(E_{J'}) $ corresponding to the vertices in $E_p^0$, where $f$ denotes the neutral element of that component,  which sends the 
canonical semigroup generators $ \mathbf{x}^p_1, \dots , \mathbf{x}^p_{N_p+1}$ of $G_p$ to the canonical set $ v^p_1,\dots , v^p_{N_p+1} $ of elements of the group
$ M(E_{J'})[f] $. 

Define a map
$$ \gamma _{J'} \colon M(J') \to M(E_{J'})$$
as follows. The map $\gamma_{J'}$ agrees with the isomorphism $\gamma_J$ when restricted to the order-ideal $M(J)$ of $M(J')$. 
The map $\gamma _{J'}$ restricted to $\AC{M(J')}{p}= \AC{M}{p}=G_p$ is just the map $\gamma _p$ above, which  sends the canonical semigroup generators of $G_p$ 
to the elements
$v^p_1, \dots , v^p_{N_p+1}$ of $\AC{M(E_{J'})}{f}$.  
In order to prove that $\gamma_{J'}$ gives a well-defined monoid homomorphism, it suffices by \cite[Corollary 1.6]{AP} to show that
if $q< p$ and $x\in \AC{M}{q}$ then $\gamma _{J'}(x) + \gamma _{J'}(p) = \gamma_{J'}(\varphi_{p,q}(x))+ \gamma _{J'}(p)$, that is, 
$\gamma _{J'}(x) + f = \gamma_{J'}(\varphi_{p,q}(x))$. 
By the same argument used in the proof of Theorem \ref{Th:FinGenRegRefAreRepr}, it is enough to consider the case where $q$ belongs to the lower cover of $p$. 
We will consider only the case where $q$ is a free prime. (The case where $q$ is a regular prime is easier and is left to the reader.)
Assume that $q$ is a free prime in the lower cover of $p$. Consider first the case where $x=q$. There is some $i$ such that $g (i) =q$ and thus $v(i) = v^q$.
Observe that, after using the relations $\mathcal{R}$, relation $\mathcal{R}$(4) can be expressed as:
$$v^p_i = v^p_i + f+ A_i+ v(i) = v^p_i+f - \gamma_{J'}(\varphi _{p,q} (q)) + v^q . $$
So, we obtain
$\gamma_{J'}(\varphi_{p,q} (q)) = f+ v^q$, that is, $f+\gamma_{J'}(x) = \gamma_{J'}(\varphi_{p,q} (x))$, as desired. 
Now suppose that $x\in M_q$. Write $x= mq+ \sum _t \varphi _{q,q'_t} (h_t) $, where $m\in \N$, $q'_t<q$, and $h_t$ is either a canonical semigroup generator of $G_{q'_t}$ (in case 
$q'_t$ is a regular prime) or $q'_t$ (in case  $q'_t$ is a free prime). Assume that we have proven that 
\begin{equation}
\label{eq:palante-patras-1}
\gamma _{J'}(\widehat{\varphi}_{p,q} \varphi_{q, q'_t} (h_t)) = f+ \gamma _{J'} (h_t)
\end{equation}
for each $t$. Then, using (\ref{eq:palante-patras-1}), the induction hypothesis and the fact that $\gamma _J$ and $\gamma _p$ are semigroup homomorphisms, we get
\begin{align*}
 \gamma _{J'} (\varphi _{p,q} (x)) & = \gamma_{J'} ( \varphi_{p,q} (mq+ \sum _t \varphi _{q,q'_t} (h_t) )) \\
 & = m\gamma_{J'} ( \varphi_{p,q} (q)) + \sum _t \gamma _{J'} (\widehat{\varphi}_{p,q} \varphi _{q,q'_t} (h_t) )\\
 & = f+ m\gamma_{J}(q) + f+ \sum _t  \gamma _{J} (h_t) \\
 & = f+ \gamma_J (mq+ \sum _t h_t ) =  f+ \gamma_J ( mq+ \sum _t \varphi _{q,q'_t} (h_t)) \\
 & = f+ \gamma_{J'} (x).
 \end{align*}
Thus, it remains to prove (\ref{eq:palante-patras-1}). Assume that $q_t'$ is a free prime, so that $h_t= q_t'$. Let $i$ be the index such that 
$g(i)= \varphi_{q,q_t'} (q_t')$ and $v(i) = v^{q'_t}= \gamma _{J'}(q_t') $. Then $\mathcal{R}$(4) gives again
$$v_i= v_i + f - \gamma _{J'}( \varphi _{p,q_t'} (q_t') ) +  \gamma _{J'}(q_t') .$$
Since $\varphi_{p,q_t'}(q_t') = \widehat{\varphi_{p,q}}\varphi_{q,q_t'} (q_t')$, we get 
$$\gamma_{J'} (\widehat{\varphi_{p,q}}\varphi_{q,q_t'} (q_t')) = f+ \gamma_{J'} (q_t') , $$
as desired. The case where $q_t'$ is a regular prime and $h_t$ is a canonical semigroup generator of $G_{q'_t}$ is treated in the same way.

This shows that the defining relations of $M(J')$ are preserved. So, the map $\gamma_{J'}$ is a well-defined homomorphism from $M(J')$ to $M(E_{J'})$.
To build the inverse $\delta_{J'}$ of $\gamma_{J'}$, we follow the idea in the proof of Proposition \ref{L:represmaps}. The image by $\delta _{J'}$
of the vertices in $E_J$ is determined by the inverse of $\gamma_J$.  
We define $\delta_{J'} (v^p_i)= \mathbf{x}^p_i$ for $1\le i \le N_p+1$, and 
$\delta_{J'} ((v^p)_i^j) = p$. It is easily checked that $\delta_{J'}$ preserves the defining relations of $M(E_{J'})$, and so it gives a well-defined 
homomorphism from $M(E_{J'})$ to $M(J')$, which is clearly the inverse of $\gamma_{J'}$.

(ii) \underline{$p$ is a free prime}: In this case, the canonical map  $ G(\varphi_p) \colon \widetilde{G}_{J_p} \to G_p$ is an almost isomorphism by hypothesis. 
Hence, $G(\varphi _p)$ is surjective and its kernel is generated by a strictly positive element $x$. Note that $J_p$ is a lower subset of $J$. 

By Lemma \ref{lem:positive-in-G(MJ)}, we can write
$$x= \sum _{i=1}^r \chi_{q_i}(n_i, g_i) + \sum _{i=r+1}^n \chi _{q_i} (g_i)$$ 
for some $n_i \in \N$, $i=1,\dots , r$,  and  $g_i\in G_{q_i}$, $i=1,\dots, n$.
By Lemma \ref{lem:Auxiliar1}, for $i\in \{ 1,\dots , r \}$ we can write
$$ g_i = \sum _{q'<q_i , q'\in I_{{\rm free}}} a_{i}^{q'} \varphi_{q_i, q'} (q') + \sum _{q''<q_i, q''\in I_{{\rm reg}}}  \sum _{j=1}^{N_{q''}+1} b_{ji}^{q''} \varphi_{q_i, q''} (\mathbf{x}^{q''}_j)  $$
for some non-negative integers $a_{i}^{q'}$ and $b^{q''}_{ji}$, and 
we can write, for $i\in \{ r+1, \dots , n \}$, 
$$ g_i = \sum _{j=1}^{N_{q_i} +1} a_{ji} \mathbf{x}_j^{q_i} , $$
for some non-negative integers $a_{ji}$.
Define the graph $E_{J'}$ with $E_{J'}^0 = E_{J}^0 \sqcup \{ v^p \}$, and with $E_{J'}^1$ the union of $E_J^1$ and a set of edges starting at $v^p$, which are determined by the following
formula:
\begin{equation}
\label{eq:arrows-for-free} 
v^p = v^p + \sum _{i=1} ^r n_i v^{q_i} + \sum_{i=1}^r ( A_i+ B_i) + \sum _{i={r+1}}^n C_i ,
\end{equation}
where 
\begin{equation}
\label{eq:arrows-for-free(2)}
A_i = \sum _{q'<q_i , q'\in I_{{\rm free}}} a_{i}^{q'} v^{q'}, \quad B_i = \sum _{q''<q_i, q''\in I_{{\rm reg}}}  \sum _{j=1}^{N_{q''}+1} b_{ji}^{q''} v^{q''}_j,\qquad (i=1,\dots , r)
\end{equation}
and 
\begin{equation}
 \label{eq:arrows-for-free(3)}
 C_i = \sum _{j=1}^{N_{q_i}+1} a_{ji} v^{q_i}_j, \qquad (i=r+1, \dots, n). 
\end{equation}
If we define $\widehat{x}:= \sum _{i=1} ^r n_i v^{q_i} + \sum_{i=1}^r ( A_i+ B_i) + \sum _{i={r+1}}^n C_i $, then we have that $\gamma_J(x)=\widehat{x}$. 

The element $v_p$ of $M(E_{J'})$ is a free prime, and so $M(E_{J'})[v_p] = \N \times G'_{v_p}$ for some abelian group $G'_{v_p}$.
It follows easily from the induction hypothesis and the form of the relation (\ref{eq:arrows-for-free}) that the map $\gamma _J|_{J}$ extends to an order-isomorphism
from $J'= J\sqcup \{ p \}$ to the set $\mathbb P$ of primes of $M(E_{J'})$, by sending $p$ to $v^p$. 
Since $M(E_{J'})$ is a finitely generated conical refinement monoid, the map $\phi_p : M(E_{J'})_{\gamma _J(J_p)} \to G'_{v_p}$ induced by the various
semigroup homomorphisms 
$$
\begin{array}{crcl}
\phi_q^p\colon & M(E_{J'})_{\gamma _J (q)}  & \rightarrow   & G'_{v^p}  \\
 & y & \mapsto  &   (v^p + y ) - v^p
\end{array}
$$ 
for $q<p$ is surjective. So, we obtain a surjective group homomorphism
$$G(\phi _p ) \colon  G(M(E_{J'})_{\gamma_J(J_p)}) \to G'_{v^p}.$$
In order to somewhat simplify the notation, we will write $M(E_{J'})_{J_p}$ instead of $M(E_{J'})_{\gamma _J(J_p)}$. 

Since $E_J^0$ is a hereditary and saturated subset of $E_{J'}^0$, the order ideal $\frak I (E_J^0)$ of $M(E_{J'})$ generated by
$E_J^0$ coincides with the monoid $M(E_J)$ (by Lemma \ref{lem:Her}), and the component $M(E_{J'})_{J_p}$ coincides with the component
$M(E_J)_{J_p}$. 
The monoid isomorphism $\gamma _J \colon M(J) \to M(E_J)$ restricts to a semigroup isomorphism 
$M_{J_p} \to M(E_J)_{J_p}$, which induces a group isomorphism
$$\widetilde{\gamma}_{J_p} \colon G(M_{J_p}) \to G(M(E_{J})_{J_p})$$
of the Grothendieck groups. 
Set $K:= \ker (G(\phi_p ))$, and notice that the relation (\ref{eq:arrows-for-free}) 
implies that $\widetilde{\gamma}_{J_p}  (x)=\widehat{x} \in K $.  

Hence, there is a commutative diagram with exact rows
\begin{equation}\label{diagram:cap-avall}
\begin{CD}
0 @>>> \langle x \rangle  @>>> G(M_{J_p}) @>{G(\varphi_p )}>> G_p @>>> 0 \\
 & &    @VVV  @VV{\widetilde{\gamma}_{J_p}}V  @VV{ \gamma_ p}V \\
0 @>>> K @>>> G(M(E_J)_{J_p}) @>{G(\phi_p )}>> G'_{v_p} @>>> 0 \, \, ,
\end{CD}
 \end{equation}
where $\gamma _p \colon G_p \to G'_{v_p}$ is the map induced from the cokernel of the inclusion $\langle x \rangle \hookrightarrow G(M_J)$ to the cokernel of the inclusion 
$K \hookrightarrow G(M(E_J)_{J_p})$. Notice that $\gamma_p$ is an onto map.

Now, we define a map
$$ \gamma _{J'} \colon M(J') \to M(E_{J'})$$
extending the monoid isomorphism $\gamma _J: M(J)\rightarrow M(E_J)$, and defining $\gamma_{J'}$ on the component $M_p\cong \N\times G_p$ of $M(J')$ by the formula 
$$\gamma_{J'}( mp + g) = m v^p + \gamma_p (g) , $$
for $m\in \N$ and $g\in G_p$. 
By \cite[Corollary 1.6]{AP}, to show that $\gamma_{J'}$ is a well-defined monoid homomorphism, it suffices  to show that
if $q< p$ and $y\in \AC{M}{q}=M_q$ then $\gamma _{J'}(y) + \gamma _{J'}(p) = \gamma_{J'}(\varphi_{p,q}(y) + p)$, that is,  
$\gamma _{J}(y) + v^p = \gamma_{p}(\varphi_{p,q}(y))+ v^p$. As $y\in \AC{M}{q}= M_q$ and we are identifying $G(M(J_p))\cong G(M_{J_p})=\widetilde{G}_{J_p}$ (Lemma \ref{lem:GrotofM}), 
there exists a composition map 
$$\tau _q\colon M_q \to M(J_p) \to G(M(J_p)) \cong G(M_{J_p})$$ 
such that $\varphi_{p,q} = G(\varphi _p) \circ \tau _q$. Analogously, we have a map 
$$\tau_{\gamma_J(q)} \colon M(E_{J'})_{\gamma_J(q)} \to  G(M(E_{J'})_{J_p}) = G(M(E_{J})_{J_p})$$ 
such that  $\phi^{v^p}_{\gamma_J(q)} = G(\phi_p ) \circ \tau _{\gamma_J(q)}$, and clearly
 $\widetilde{\gamma} _{J_p} \circ \tau_q = \tau_{\gamma _J (q)} \circ \gamma_J|_{M_q} $. 
 
Using this fact, and the commutativity of (\ref{diagram:cap-avall}), we have that
\begin{align*}
 \gamma _p (\varphi_{p,q} (y)) + v^p & = \gamma _p (G(\varphi _p) (\tau_q  (y))) + v^p \\
 &  =  G(\phi_p) ( \widetilde{\gamma} _{J_p} (\tau _q (y))) + v^p \\
& =  G(\phi_p) ( \tau_{\gamma_J (q)} (\gamma_J(y))) + v^p \\
& = \phi^{v^p}_{ \gamma_J(q)} (\gamma _J(y)) + v^p  \\
 & = ((v^p + \gamma_J (y) )-v^p) + v^p  \\
 & = v^p + \gamma _J (y) \, ,
\end{align*}
as desired. 

This shows that there is a well-defined monoid homomorphism $\gamma_{J'} \colon M(J') \to M(E_{J'})$ sending the canonical semigroup generators of $M(J')$ to the corresponding canonical 
sets of vertices
seen in $M(E_{J'})$. In particular, $\gamma_{J'} $ is an onto map. 

In order to prove the injectivity of $\gamma_{J'} $, we can build an inverse map $\delta_{J'}:  M(E_{J'}) \to M(J')$, 
as follows.  On $M(E_{J})$ we define $\delta_{J'}$ to be $\gamma_J^{-1}$, while $\delta_{J'}(v_p):=p$. Notice that the only 
relation on $M(E_{J'})$ not occurring already in $M(E_J)$ is $v^p=v^p+\widehat{x}$, where $\gamma_J(x)=\widehat{x}$. Thus, $\delta_{J'}(\widehat{x})=x$. 
But $x$ generates the kernel of the map
$$G(\varphi_p): \widetilde{G}_{J_p}\rightarrow G_p\hookrightarrow \widehat{G}_p=\Z\times G_p,$$
so that $(p+x)-p$ equals $0$ in $G_p$. Hence, the relation $p=p+x$ holds in $M(J')$. Thus, $\delta_{J'}$ is a well-defined monoid homomorphism, and it is the inverse of $\gamma_{J'}$. This completes the proof of the inductive step, and so the result holds, as desired.
\end{proof}

As an illustration of Theorem \ref{thm:main}, we obtain the characterization of antisymmetric finitely generated graph monoids \cite{APW08}.

\begin{corol} \cite[Theorem 5.1]{APW08}
 \label{cor:APW}
 Let $M$ be a finitely generated antisymmetric refinement monoid. Then $M$ is a graph monoid if and only if for each free prime $p$ in $M$ the lower cover of $p$ contains at most one free prime.
 \end{corol}

\begin{proof}
 Note that in the antisymmetric case we have that the groups $G_q$ are trivial for all primes $q$. Consequently the maps $\varphi _{p,q}$, for $q<p$ are all the zero map.
 
 Let $p$ be a free prime of $M$. If $q_1,\dots , q_r$ are the free primes in the lower cover of $p$, then the component $M_ {J_p}$ is isomorphic to $\N ^r$ and thus 
 $G(M_{J_p}) = \Z^r$. The canonical map 
 $$G(\varphi_p) \colon G(M_{J_p}) \longrightarrow G_p$$ reduces to the trivial map $\Z^r\longrightarrow  0$. If $r>1$ the kernel of this map cannot be cyclic. 
 Conversely if $r\le 1$ the kernel of this map is generated by a strictly positive element coming from $\sum_{q\in \rL (I,p)} q$. 
 \end{proof}

\section{Concluding remarks}

We finish the paper with a short overview of the current status of the realization problem for von Neumann regular rings, with special emphasis on the impact of the present paper
in this program.

After the work done in this paper, we obtain the realization of a large amount of finitely generated 
conical refinement monoids as graph monoids, and thus, using the results in \cite{AB}, we can realize those monoids as $\mathcal V$-monoids of 
suitable von Neumann regular rings (see Corollary \ref{cor:realization}). Moreover, we understand the reason why some finitely generated conical 
refinement monoids 
are not graph monoids. In order to obtain a general realization result for all finitely generated conical refinement monoids, it would be helpful to find a 
more comprehensive class of algebras, extending the class of Leavitt path algebras, which should include in particular the class of algebras considered in \cite{Aposet}
associated to finite posets. A natural candidate for such a class could be the class of Kumjian-Pask algebras associated to higher rank graphs \cite{ACHR}, or some variant
of it.  

In \cite{AGSemigForum}, the first-named author and Ken Goodearl have introduced a fundamental distinction between refinement monoids.
A refinement monoid $M$ is said to be {\it tame} in case it is the direct limit of a directed system of finitely generated refinement monoids, 
and $M$ is said to be {\it wild} otherwise. Tame refinement monoids inherit from the finitely generated refinement monoids which approximate them  
good cancellation and decomposition properties, and include large classes of refinement monoids (see \cite[Section 3]{AGSemigForum}).
Every graph monoid $M(E)$ of an arbitrary graph $E$ is tame (\cite[Theorem 4.1]{AGSemigForum}).  
Also, every primely generated conical refinement monoid is tame (\cite[Theorem 0.1]{AP}), but not every graph monoid of a row-finite graph
is primely generated (see \cite[p. 171]{AMFP}). This raises the question of determining those tame conical refinement monoids which are graph monoids.
Even assuming that the monoids are countable, the techniques introduced in the present paper are not enough for this purpose, a first difficulty 
(but not the only one) is that we cannot assume that the monoids are primely generated. 

For the monoid $M:=\langle p, a, b : p=p+a,\,  p=p+b\rangle $ considered in Example \ref{Exam:NoGraphMonoid}, it was proven in \cite[Theorem 4.1]{APW08} that
$M$ is not even a direct limit of graph monoids. However, $M$ has been realized as the $\mathcal V$-monoid of a von Neumann regular $K$-algebra for any field $K$,
see \cite[Example 3.3]{Aposet}. 

We present here some examples of conical refinement monoids which can be shown to be graph monoids by using our results.

\begin{prop}
 \label{finite-monoids} 
 Let $M$ be a conical finite refinement monoid. Then $M$ is a graph monoid.  
 \end{prop}

\begin{proof}
 Every prime of $M$ is necessarily regular, so $M$ is regular and the result follows from Theorem \ref{Th:FinGenRegRefAreRepr}. 
 \end{proof}

 \begin{exem}
  \label{exam:free} {\rm We consider a family of $I$-systems (where $I$ is a fixed poset) closely related to the ones introduced in Example \ref{exam:Easy-I-system}.  
  Let $$I= \{p,q_1,q_2,\dots , q_r \},$$ where $p>q_i$ for all $i$, and all $q_i$ are pairwise incomparable.  
We set $I= I_{{\rm free}}$. Since $q_i$ are minimal free primes we must have $G_{q_i}=\{e_{q_i}\}$.  
If $r=1$, then $G_p$ must be a finite cyclic group (Example \ref{exam:Easy-I-system}(1)), and 
the criteria in Theorem \ref{thm:main} are then satisfied, so $M(\mathcal J)$ is a graph monoid. If $r>1$, then we see from Theorem \ref{thm:main} that $M(\mathcal J)$ is a graph monoid 
if and only if $G_p$ is of the form $\Z^r/\langle (n_1,n_2,\dots , n_r)\rangle $, where $n_i$ are strictly positive integers for all $i$, and the maps $\varphi_{p,q_i}$ are the obvious maps
$\N \to \Z^r/\langle (n_1,n_2,\dots , n_r)\rangle $ sending the generator of $\N$ to the image of the canonical basis vector $e_i$. In terms of the examples given in Example \ref{exam:Easy-I-system}(2), we see that 
they give rise to a graph monoid if and only if $s=r-1$.} 
\end{exem}

 Finally we want to add some words about the situation regarding the realization problem for wild refinement monoids. The 
{\it abelianized Leavitt path algebras} $L^{{\rm ab}}(E,C)$ associated to separated graphs $(E,C)$ (see \cite{AE}, and also \cite{AG12})
have the property that their monoids $\mathcal V (L^{{\rm ab}}(E,C))$ are refinement monoids which may lack any possible cancellation or order-cancellation property 
which can fail in a conical abelian monoid. In particular, the monoids $\mathcal V (L^{{\rm ab}}(E,C))$ are often wild refinement monoids.
Some particular examples of this sort have been examined in \cite{AGtrans}, showing that they can be realized by exchange $K$-algebras (\cite[Theorem 4.10]{AGtrans}) for an arbitrary field $K$, and by 
von Neumann regular $K$-algebras exactly when the field is countable (\cite[Theorem 5.5]{AGtrans}).  This shows that there are evident differences with the case of tame refinement monoids, 
and suggests that different techniques have to be developed to deal with the case of wild refinement monoids.


\section*{Acknowledgments}

Part of this work was done during visits of the second author to the
Departament de Matem\`atiques de la Universitat Aut\`onoma de Barcelona (Spain). The second author thanks the center for its kind hospitality. 
Both authors thank Kevin O'Meara and the referee for various suggestions which helped to improve the original version of this paper.


\end{document}